\documentclass{rmmcart}
\usepackage[margin = 2cm]{geometry}
\usepackage[utf8]{inputenc}
\usepackage{amsmath,amssymb,amsthm,color,verbatim}
\usepackage{booktabs}
\usepackage{mathrsfs}
\usepackage[
  separate-uncertainty = true,
  multi-part-units = repeat
]{siunitx}
\usepackage{mathtools}
\usepackage{enumerate}
\usepackage{bbm}
\usepackage{tikz}
\usepackage{tikz-cd}
\usepackage{xypic}
\usepackage{ytableau}

\newcommand{\N}{\mathbb{N}}
\newcommand{\Z}{\mathbb{Z}}
\newcommand{\Q}{\mathbb{Q}}
\newcommand{\F}{\mathbb{F}}
\newcommand{\K}{\mathcal{K}}
\renewcommand{\k}{\mathbbm{k}}

\newcommand{\Groth}{{R}}

\newcommand{\boldx}{{\mathbf{x}}}
\newcommand{\bolda}{{\mathbf{a}}}
\newcommand{\boldb}{{\mathbf{b}}}
\newcommand{\boldc}{{\mathbf{c}}}
\newcommand{\boldt}{{\mathbf{t}}}

\newcommand{\supp}{{\mathrm{supp}}}
\newcommand{\opp}{{\mathrm{opp}}}
\newcommand{\Tor}{{\mathrm{Tor}}}
\newcommand{\Aut}{{\mathrm{Aut}}}
\newcommand{\Hilb}{{\mathrm{Hilb}}}
\newcommand{\Sd}{{\mathrm{Sd}}}
\newcommand{\Des}{{\mathrm{Des}}}
\newcommand{\maj}{{\mathrm{maj}}}
\newcommand{\sgn}{{\mathrm{sgn}}}
\newcommand{\filt}{{\mathrm{filt}}}
\newcommand{\depth}{{\mathrm{depth}}}
\newcommand{\equivariant}{{\mathrm{eq}}}

\newtheorem*{main-theorem}{Main Theorem}
\newtheorem{thm}{Theorem}[section]

\newtheorem{prop}[thm]{Proposition}
\newtheorem{cor}[thm]{Corollary}
\newtheorem{conjecture}[thm]{Conjecture}
\newtheorem{question}[thm]{Question}
\newtheorem{defn}[thm]{Definition}

\newtheorem{example}[thm]{Example}
\newtheorem{remark}[thm]{Remark}

\begin{document}

\title[Colorful Hochster formula and universal parameters for face rings]{A colorful Hochster formula and universal parameters for face rings}

\author{Ashleigh Adams}
\email{adams869@umn.edu}

\author{Victor Reiner}
\email{reiner@umn.edu}
\address{School of Mathematics\\University of Minnesota\\Minneapolis, MN 55455}

\thanks{Work of first, second authors supported by NSF Graduate Research Fellowship
and NSF grant DMS-1601961, respectively}
\keywords{Hochster formula, Stanley-Reisner, simplicial poset, depth, balanced, symmetry}
\subjclass{13F55,13F50,13D02}

\begin{abstract}
  This paper has two related parts.  The first generalizes Hochster's formula on resolutions of Stanley-Reisner rings to a colorful version, applicable to any proper vertex-coloring of a simplicial complex.  The second part examines a universal system of parameters for Stanley-Reisner rings of simplicial complexes, and more generally, face rings of simplicial posets.  These parameters have good properties, including being fixed under symmetries, and detecting depth of the face ring.  Moreover, when resolving the face ring over these parameters, the shape is predicted, conjecturally, by the colorful Hochster formula.
\end{abstract}

\maketitle


\section{Introduction}
\label{intro-section}

This paper has two closely related parts, concerned with resolutions of Stanley-Reisner rings of simplicial complexes and face rings of simplicial posets as defined by Stanley in
\cite{Stanley-fvec-hvec}.

\subsection{Part 1. Stanley-Reisner rings}
The first part deals with the {\it Stanley-Reisner ring} $k[\Delta]$ for an abstract simplicial complex $\Delta$ on vertex set $V=[n]:=\{1,2,\ldots,n\}$. Recall that 
$$
\k[\Delta]:=\k[x_1,\ldots,x_n]/I_\Delta
$$
where the ideal $I_\Delta$ is
the $\k$-linear span of all monomials not supported on a face of $\Delta$.  

Assume that one is given a 
map $\kappa: V \rightarrow [d]$ which
is a {\it proper vertex $d$-coloring} of $\Delta$ in the sense that every edge $\{i,i'\}$ of $\Delta$ has $\kappa(i) \neq \kappa(i')$.
Section~\ref{colorful-Hochster-section} below discusses how this endows $\k[\Delta]$ with an $\N^d$-multigrading,
in which $\deg(x_i)$ is the standard basis vector $\epsilon_{\kappa(i)}$ in $\N^d$.  It is also shown there that
$\k[\Delta]$, is a finitely-generated $\N^d$-graded module over the polynomial ring $A:=\k[z_1,\ldots,z_d]$
via a ring map
$$
\begin{array}{rcl}
A &\longrightarrow &\k[\Delta]\\
z_j &\longmapsto &\gamma_j:=\sum_{i \in \kappa^{-1}(j)} x_i \text{ for }j=1,2,\ldots,d.
\end{array}
$$
The shape of the minimal free resolution of $\k[\Delta]$ as an
$A$-module is described by our first main result,
a {\it colorful Hochster formula} (Theorem~\ref{colorful-Hochster-formula}),
generalizing a celebrated formula of
Hochster \cite[Thm.~5.1]{Hochster} for the case $d=n$ with trivial coloring $\kappa$
assigning each vertex a different color.
Our formula asserts that,
for $\boldb$ in $\N^d$,
the $\boldb$-multigraded component of
$\Tor_*^A(\k[\Delta],\k)$ vanishes unless $\boldb$ lies in
$\{0,1\}^d$, so $\boldb=\sum_{j \in S} \epsilon_j$ for a subset $S \subseteq [d]$, in which case
$$
\Tor_m^A(\k[\Delta],\k)_{\boldb} \cong \tilde{H}^{\#S-m-1}(\Delta|_S,\k).
$$
Here $\tilde{H}^*(-,\k)$ denotes reduced simplicial cohomology with
coefficients in $\k$, and $\Delta|_S$ is the {\it $S$-color-selected
subcomplex} of $\Delta$, consisting of its simplices whose vertices all have $\kappa$-coloring lying in $S$.

\subsection{Part 2. Face rings and universal parameters}
The second part of this paper connects the colorful Hochster formula with our original motivation: to better understand the {\it face rings} associated by Stanley to
what he called {\it simplicial posets}, along with their symmetries.  These are
posets having a unique bottom
element in which all lower intervals are isomorphic to Boolean algebras.  Each simplicial poset $P$ is 
the face poset of an associated regular CW-complex
$\Delta$, generalizing an abstract simplicial complex.  Stanley associated to each of them a {\it face ring}
$\k[\Delta]=S/J_\Delta$ 
generalizing the Stanley-Reisner ring; see 
Section~\ref{simplicial-poset-section} below.  Here $S$ is a polynomial ring having a variable $y_F$ for each non-empty face $F$ of $\Delta$ (with convention that the empty face $\varnothing$ has $y_\varnothing:=1$), while $J_\Delta$ is the ideal generated by two kinds of
quadratic relations:
one sets $y_F y_{F'}=0$ in $\k[\Delta]$ for faces $F,F'$ having no face $G$ containing both of them, and otherwise
$$
y_F y_{F'} =  y_{F \cap F'} \sum_G y_G
$$
where the sum
is over faces $G$ in $\Delta$ which are minimal among those containing both $F,F'$.
When $\Delta$ is
actually a simplicial complex, the above face ring
is isomorphic to the usual Stanley-Reisner ring for $\Delta$, via the map sending $y_F \mapsto \prod_{i \in F} x_i$ to the product of variables corresponding to vertices of $F$.

We were originally motivated to study the face ring $\k[\Delta]$ for any
such complex $\Delta$ as a graded representation of the group of (cellular) automorphisms of $\Delta$.  A helpful feature in this regard is a certain {\it universal system of parameters}, discussed in
Section~\ref{universal-parameters-section}, that has appeared in work of De Concini, Eisenbud and Procesi on algebras with straightening laws \cite{Hodge-algebras}, work of Garsia and Stanton on invariant theory of permutation groups \cite{Garsia-Stanton}, work of D.E. Smith on sheaves on posets \cite{Smith}, and most recently work of Herzog and Moradi \cite{HerzogMoradi}.   The face ring $\k[\Delta]$ has Krull dimension $d$ when $\Delta$ has topological dimension $d-1$, and the sequence of elements $\Theta=(\theta_1,\ldots,\theta_d)$ defined by 
$$
\theta_j:=\sum_{\substack{\text{faces }F \in \Delta:\\ \dim(F)=j-1}} y_F
$$
turn out to give a universal system of parameters, fixed pointwise by any
cellular automorphism of $\Delta$.  

Generalizing work of D.E. Smith, Theorem~\ref{depth-sensitivity-theorem} will show that these parameters $\Theta$
detect {\it depth} of $\k[\Delta]$:
$$
\depth \,  \k[\Delta] =\max\{\delta: (\theta_1,\theta_2,\ldots,\theta_\delta) \text{ forms a regular sequence on }\k[\Delta]\}.
$$

We then go on to conjecture (Conjecture~\ref{canonical-resolution})
the shape of the $\N$-graded minimal resolution of the face ring $\k[\Delta]$ over the universal parameter ring $\k[\Theta]=\k[\theta_1,\ldots,\theta_d]$, and connect it to the colorful Hochster formula from Part 1.  Because $\k[\Delta]$ is an {\it algebra with straightening law} \cite{Hodge-algebras} over the face poset of $\Delta$,
it may be regarded as a {\it Gr\"obner deformation} of
the Stanley-Reisner ring $\k[\Sd \Delta]$ for the {\it barycentric subdivision} $\Sd \Delta$.  This subdivision has a
canonical proper vertex $d$-coloring $\kappa$ which assigns color $j$ to the barycenter vertex of each $(j-1)$-dimensional face.  Therefore, as in the first part of
this paper, $\k[\Sd \Delta]$ has a minimal free 
resolution over a ``colorful" parameter ring 
$A=\k[\Gamma]=\k[\gamma_1,\ldots,\gamma_d]$, and the $\N^d$-graded
resolution Betti numbers are predicted by the colorful Hochster formula Theorem~\ref{colorful-Hochster-formula}. 
The universal parameter ring $\k[\Theta]$ for 
$\k[\Delta]$ maps to this colorful parameter ring $\k[\Gamma]$ for $\k[\Sd \Delta]$ under the Gr\"obner deformation. Conjecture~\ref{canonical-resolution} asserts that, after specializing the $\N^d$-multigrading of $\k[\Sd \Delta]$ via the map
$
\N^d \rightarrow \N
$ sending
$
\epsilon_j \mapsto j,
$
the $\N$-graded Betti numbers are equal:
\begin{equation}
    \label{conjecture-paraphrased}
\begin{aligned}
\Tor_m^{\k[\Theta]}(\k[\Delta],\k)_j 
\cong
\Tor_m^{\k[\Gamma]}(\k[\Sd \Delta],\k)_j
\cong
\bigoplus_{\substack{S \subseteq [d]:\\ j=\sum_{s \in S}s} }
\tilde{H}^{\#S-m-1}((\Sd \Delta)|_S,\k).
\end{aligned}
\end{equation}
In fact, it was the form of the right side of \eqref{conjecture-paraphrased} in examples that
led us to the formulation of Theorem~\ref{colorful-Hochster-formula}.

\begin{remark} \rm 
The authors thank Patricia Klein for pointing out that, since $\k[\Sd \Delta]$ is a {\it square-free Gr\"obner deformation} of $\k[\Delta]$, Conjecture~\ref{canonical-resolution} is in the spirit of a conjecture of Herzog, proven by Conca and Varbaro \cite{Conca-Varbaro}, concerning preservation of
extremal Betti numbers under square-free Gr\"obner deformations. 
It is unclear why {\it all Betti numbers} would be preserved in this case.
\end{remark}

The rest of the paper is structured as follows.  

Section~\ref{Stanley-Reisner-section} reviews material on Stanley-Reisner rings, introduces their proper vertex-colorings, and discusses the group of color-preserving symmetries, as well as Hilbert series, $f$-vectors and $h$-vectors that take this symmetry into account.  It also discusses {\it order complexes of posets}, which naturally come with a proper vertex-coloring, including some of our motivating examples with large groups of symmetries.

Section~\ref{colorful-Hochster-section} states and proves the colorful Hochster formula, Theorem~\ref{colorful-Hochster-formula}.

Section~\ref{simplicial-poset-section} reviews simplicial posets and their face rings, including their relationship to algebras with straightening laws, and Gr\"obner deformations.

Section~\ref{universal-parameters-section} explains why the universal parameters
$\Theta$ really are a system of parameters for the face ring $\k[\Delta]$,
and proves that they detect its depth in Theorem~\ref{depth-sensitivity-theorem}.

Section~\ref{main-conjecture-section} states Conjecture~\ref{canonical-resolution} on the $\k[\Theta]$-resolution
of $\k[\Delta]$, and indicates some evidence in its favor.

\section{Stanley-Reisner review and set-up}
\label{Stanley-Reisner-section}

\subsection{Stanley-Reisner rings}

Let $\Delta$ be an {\it abstract simplicial complex} on a
finite vertex set 
$$
V=[n]:=\{1,2,\ldots,n\},
$$
meaning that $\Delta$ is a collection of subsets $F \subset [n]$ called {\it faces}, with the property that whenever $F$ lies in $\Delta$, then any subset $F' \subseteq F$ also lies in $\Delta$.  

A face $F$ in $\Delta$ has {\it dimension} $\dim(F):=\#F-1$.  Zero- and one-dimensional faces are called {\it vertices} and {\it edges}, respectively.
The dimension  $\dim(\Delta):=\max\{\dim(F): F \in \Delta\}$.
Say that $\Delta$ is {\it pure} if all
of its maximal faces have the dimension, namely $\dim(\Delta)$.

Fix a field $\k$, and
let $\k[\boldx]:=\k[x_1,\ldots,x_n]$ be the polynomial ring
in variables indexed by the vertices $V=[n]$.  For a vector $\bolda=(a_1,\ldots,a_n)$ in $\N^n$, we use multi-index notation
for monomials 
$
\boldx^{\bolda}:=x_1^{a_1} \cdots x_n^{a_n}.
$
Letting $e_1,\ldots,e_n$ be standard basis vectors in $\Z^n$,
the square-free monomial indexed by $S \subseteq [n]$ is
$$
\boldx^S:=\prod_{i \in S} x_i = \boldx^{\sum_{i \in S} e_i}.
$$

\begin{defn} \rm
For a simplicial complex $\Delta$ on vertices $V=[n]$,
the {\it Stanley-Reisner ring} $\k[\Delta]$ is 
$$
\k[\Delta]:=\k[\boldx]/I_\Delta
$$
where the {\it Stanley-Reisner ideal} $I_\Delta$ is generated by all square-free monomials
$\boldx^S$ with $S$ {\it not} in $\Delta$.
\end{defn}

\noindent
It is easily seen that $\k[\Delta]$
has $\k$-basis the monomials $\boldx^\bolda$ with
{\it support set} 
$
\supp(\bolda):=\{i: a_i >0\}
$
in $\Delta$.

\subsection{Vertex-colorings}

\begin{defn} \rm
A {\it (proper, vertex-) $d$-coloring} of $\Delta$ is a map
$V \overset{\kappa}{\rightarrow} [d]$ such that the vertices
in any face $F$ in $\Delta$ have $\#F$ distinct colors, that is,
$\#\kappa(F)=\#F$.  Equivalently, $\kappa(i) \neq \kappa(j)$
for all edges $\{i,j\}$ in $\Delta$.
\end{defn}

There are two extreme cases of such colorings:
\begin{itemize}
    \item The {\it trivial} $n$-coloring $\kappa$ is the identity map $V=[n] \rightarrow [n]$ assigning every vertex its own color.
    \item A {\it balanced} $d$-coloring is a proper coloring $\kappa$ with $d=\dim(\Delta)+1$, which may or may not exist;  when one does exist then $\Delta$ is called a {\it balanced} simplicial complex.  
\end{itemize}

Given a $d$-coloring $\kappa$ of $\Delta$, one can endow
$\k[\boldx]$ with an $\N^d$-multigrading in which $\deg(x_i):=\epsilon_{\kappa(i)}$, where $\epsilon_j$ is the $j^{th}$ standard
basis vector in $\Z^d$.  One can check that 
the Stanley-Reisner ideal $I_\Delta$ is homogeneous with respect to this 
$\N^d$-grading, and hence this induces 
an $\N^d$-multigrading on $\k[\Delta]=\k[\boldx]/I_\Delta$.

\subsection{Symmetries}

Because our motivation was originally representation-theoretic\footnote{We hope the representation-theoretic baggage does not greatly annoy readers interested solely in
Stanley-Reisner rings. Such readers can safely ignore discussions involving
the phrases {\it symmetry}, {\it equivariant}, and {\it Grothendieck ring}.},
we wish to incorporate the action on all of these objects of
a subgroup of the simplicial automorphism group $\Aut(\Delta)$,
namely the subgroup of {\it color-preserving automorphisms}
$$
\Aut_\kappa(\Delta):=
\{ g \in \Aut(\Delta): \kappa(g(i))=\kappa(i)\text{ for all }i\text{ in }V=[n]\}.
$$
This group acts on $\k[\Delta]$ preserving the $\N^d$-multigrading.  Thus, for
each fixed multidegree $\boldb$ in $\N^d$, the {\it
$\boldb$-homogeneous component} of $\k[\Delta]$, denoted $\k[\Delta]_\boldb$, is not only a $\k$-vector
space, but also a representation of the group $\Aut_\kappa(\Delta)$, or a module over the {\it group algebra} $\k[\Aut_\kappa(\Delta)]$. To keep track of these
representations with fields $\k$ of any characteristic, it is 
convenient to introduce a certain {\it Grothendieck ring}.

\begin{defn} \rm
For a finite group $G$ (such as any subgroup $G$ of $\Aut_\kappa(\Delta)$),
define the {\it Grothendieck ring} $\Groth_{\k}(G)$ of virtual $\k G$-modules first as an abelian group: $\Groth_{\k}(G)$ 
is the quotient of the free $\Z$-module
having basis elements $[U]$ for each $\k G$-module $U$,
in which one mods out by the relations
\begin{itemize}
    \item $[U] = [U']$ if $U \cong U'$ as $\k G$-modules, and
    \item $U_2=U_1+U_3$
when $0 \rightarrow U_1 \rightarrow U_2 \rightarrow U_3 \rightarrow 0$
is a short exact sequence of $\k G$-modules.
\end{itemize}
Then ring multiplication in $\Groth_{\k}(G)$ is induced from $[U] \cdot [U']:=[U \otimes U']$, which descends to the quotient.
\end{defn}

The Jordan-H\"older theorem implies that $\Groth_{\k}(G)$
is a free $\Z$-module, with a $\Z$-basis given by the classes
$\{[U_1],\ldots,[U_t]\}$ of the inequivalent simple $\k G$-modules $U_i$.
Among these is the class of the {\it trivial} one-dimensional module $\k$,
on which every $g$ acts as the identity;  the class of this trivial module is the multiplicative
identity in $\Groth_{\k}(G)$, and will therefore be denoted by $1$.

Equivariant assertions that involve $\Groth_{\k}(G)$
can always be specialized to non-equivariant ones that ignore the $\k G$-module structure, by applying the {\it dimension homomorphism},
a ring map defined as follows:
\begin{equation}
\label{dimension-homomorphism}
\begin{array}{rcl}
\Groth_{\k}(G) &\overset{\dim}{\longrightarrow}& \Z \\[0in]
 [U] &\longmapsto& \dim_\k U.
\end{array}
\end{equation}

\subsection{Hilbert series and equivariant Hilbert series}

Let $\Delta$ be a simplicial complex $\Delta$ with a proper $d$-coloring $\kappa$, and $G$ a subgroup of $\Aut_\kappa(\Delta))$.
One can then keep track of the {\it $\N^d$-graded Hilbert
series} lying in 
$\Z[[\boldt]]:=\Z[[t_1,\ldots,t_d]]$,
and more generally its {\it equivariant Hilbert series} lying in
$\Groth_{\k}(G)[[\boldt]]$:
$$
\begin{aligned}
\Hilb(\k[\Delta],\boldt)&
  :=\sum_{\boldb \in \N^d} \dim_\k \k[\Delta]_\boldb \cdot \boldt^\boldb \\
  \Hilb_{\equivariant}(\k[\Delta],\boldt)&
  :=\sum_{\boldb \in \N^d} [ \k[\Delta]_\boldb ] \cdot \boldt^\boldb
\end{aligned}
$$

To write down formulas for these Hilbert series,
we introduce the following notions.

\begin{defn} \rm
For any proper $d$-coloring $\kappa$ of $\Delta$, define the 
{\it $\kappa$-flag $f$-vector} $(f^\kappa_S)_{S \subseteq [d]}$ with entries
$$
f^\kappa_S(\Delta)= \#\{F\in \Delta: \kappa(F)=S \},
$$
and the {\it $\kappa$-flag $h$-vector} $(h^\kappa_S)_{S \subseteq [d]}$
with entries
$$
h^\kappa_S(\Delta):=\sum_{T: T \subseteq S} (-1)^{S \setminus T} f^\kappa_T(\Delta)
$$
or equivalently, via inclusion-exclusion
$$
f^\kappa_S(\Delta):=\sum_{T: T \subseteq S}  h^\kappa_T(\Delta).
$$
More generally, define $[f^\kappa_S(\Delta)]$ in  $\Groth_{\k}(G)$ to be the class of the $G$-permutation representation on the set 
$$
\{F\in \Delta: \kappa(F)=S \},
$$ 
or the sum of the coset representations for the stabilizer subgroups of orbit representatives
of this set. Then define the element $[ h^\kappa_S(\Delta)]$ as follows
(cf. Stanley \cite[\S 1]{Stanley-some-aspects}):
\begin{equation}
    \label{equivariant-h-definition}
[ h^\kappa_S(\Delta) ]
:=\sum_{T: T \subseteq S} (-1)^{\#S- \#T} [f^\kappa_S(\Delta)]
=(-1)^{\#S-1} \tilde{\chi}_\equivariant(\Delta|_S)
\end{equation}
where $\tilde{\chi}_\equivariant(\Delta|_S)$ is the {\it (equivariant) reduced Euler characteristic}
\begin{equation}
\label{Euler-Poincare}
\tilde{\chi}_\equivariant(\Delta|_S)
=\sum_{i \geq -1} (-1)^i [\tilde{C}^i(\Delta|_S,\k)]
= \sum_{i \geq -1} (-1)^i [\tilde{H}^i(\Delta|_S,\k)].
\end{equation}
for the {\it color-selected subcomplex} 
\begin{equation}
    \label{color-selected-subcomplex-defn}
\Delta|_S:=\{ F \in \Delta: \kappa(F) \subseteq S\}.
\end{equation}
Of course,
applying the dimension homomorphism \eqref{dimension-homomorphism} to
$[ f^\kappa_S(\Delta) ], [ h^\kappa_S(\Delta) ]$
recovers their non-equivariant versions, that is,
$
f^\kappa_S(\Delta)=\dim[ f^\kappa_S(\Delta)]
$
and
$
h^\kappa_S(\Delta)=\dim[ h^\kappa_S(\Delta)].
$
\end{defn}

\noindent
The next proposition generalizes formulas of
Stanley \cite[p. 54]{Stanley-littlegreenbook}, and Garsia and Stanton \cite[eqn. (0.8)]{Garsia-Stanton}.
\begin{prop}
\label{Hilbert-series-calculations}
Given any $d$-coloring $\kappa$ of a simplicial complex $\Delta$,
one has the following expressions for the $\N^d$-graded equivariant Hilbert
series
\begin{equation}
\label{equivariant-Hilb-expressions}
\begin{aligned}
\Hilb_{\equivariant}(\k[\Delta],\boldt)
  &=\sum_{S \subseteq [d]} \frac{[f^\kappa_S(\Delta)] \cdot \boldt^S}{\prod_{j \in S}(1-t_j)} 
  =\frac{1}{\prod_{j=1}^d (1-t_j)}
   \sum_{S \subseteq [d]} [h^\kappa_S(\Delta)] \cdot \boldt^S
\end{aligned}
\end{equation}
and non-equivariant versions
$$
\begin{aligned}
\Hilb(\k[\Delta],\boldt)
  &=\sum_{S \subseteq [d]} \frac{f^\kappa_S(\Delta) \cdot \boldt^S}{\prod_{j \in S}(1-t_j)} 
  =\frac{1}{\prod_{j=1}^d (1-t_j)}
   \sum_{S \subseteq [d]} h^\kappa_S(\Delta) \cdot \boldt^S
\end{aligned}
$$
\end{prop}
\begin{proof}
It suffices to prove \eqref{equivariant-Hilb-expressions}, and apply the dimension homomorphism
\eqref{dimension-homomorphism} to deduce the non-equivariant versions.  The first equality in \eqref{equivariant-Hilb-expressions} comes from observing that a face $F \in \Delta$ with colors $\kappa(F)=S$ has
$$
\sum_{\substack{\text{monomials }m:\\ \supp(m)=F}} \boldt^{\deg_{\N^d}(m)} 
= \prod_{j \in S} (t_j+t_j^2+\cdots)
= \prod_{j \in S} \frac{t_j}{1-t_j}
= \frac{\boldt^S}{\prod_{j \in S}(1-t_j)}.
$$
The second equality in \eqref{equivariant-Hilb-expressions} puts the sum over the common denominator
$\prod_{j=1}^d (1-t_j)$ with this numerator:
$$
\sum_{S \subseteq [d]} [f^\kappa_S(\Delta)] \cdot \boldt^S \prod_{j \in [d] \setminus S}(1-t_j) 
 =\sum_{S \subseteq [d]} [f^\kappa_S(\Delta)] 
    \sum_{T: S \subseteq T \subseteq [d]}
             (-1)^{\#S-\#T} \boldt^T\\
 =\sum_{T\subseteq [d]}
               [h^\kappa_T(\Delta)] \cdot \boldt^T. \qedhere
$$
\end{proof}

\begin{example} 
\label{running-example-one}
\rm 
Consider this two-dimensional simplicial complex $\Delta$ on vertex set $V=[8]$:

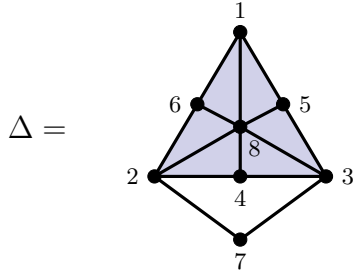
\begin{center}
\begin{tikzpicture}[scale=0.6]
\pgfmathsetlengthmacro{\rad}{25pt};
\pgfmathsetlengthmacro{\vertsize}{\rad*0.1666};

\pgfmathsetmacro{\triangleopacity}{0.4};
\definecolor{dimtwocolor}{RGB}{140,140,200}
    
    \begin{scope}[shift={(-4.5,0)}]
        \node[] at (0,0){\Large $\Delta=$};
    \end{scope}
     
    \begin{scope}{[shift={(0,0)}]}
    
     	\coordinate (1) at (0,2.1);
		\coordinate (2) at (-1.9,-1.1);
		\coordinate (3) at (1.9,-1.1);
		
     	\coordinate (4) at (0,-1.1);
		\coordinate (5) at (.95,.5);
		\coordinate (6) at (-.95,.5);
		
		\coordinate (7) at (0,-2.5); 
		\coordinate (8) at (0,0); 

		\draw [very thick, draw=black, fill=dimtwocolor, fill opacity=\triangleopacity] (1) -- (6) -- (2) -- (4) -- (3) -- (5) -- cycle;

		\draw [very thick, draw=black, fill=dimtwocolor, fill opacity=\triangleopacity] (8) -- (1);		
		\draw [very thick, draw=black, fill=dimtwocolor, fill opacity=\triangleopacity] (8) -- (2);
		\draw [very thick, draw=black, fill=dimtwocolor, fill opacity=\triangleopacity] (8) -- (3);
		\draw [very thick, draw=black, fill=dimtwocolor, fill opacity=\triangleopacity] (8) -- (4);
		\draw [very thick, draw=black, fill=dimtwocolor, fill opacity=\triangleopacity] (8) -- (5);
		\draw [very thick, draw=black, fill=dimtwocolor, fill opacity=\triangleopacity] (8) -- (6);

		\draw [very thick, draw=black] (2) to 
		(7) to 
		(3);

		\draw [black, fill=black] (1) circle [radius=\vertsize];
		\draw [black, fill=black] (2) circle [radius=\vertsize];
		\draw [black, fill=black] (3) circle [radius=\vertsize];
		\draw [black, fill=black] (4) circle [radius=\vertsize];
		\draw [black, fill=black] (5) circle [radius=\vertsize];
		\draw [black, fill=black] (6) circle [radius=\vertsize];
		\draw [black, fill=black] (7) circle [radius=\vertsize];
		\draw [black, fill=black] (8) circle [radius=\vertsize];
		
		\node[yshift=2*\vertsize]                        at (1) {\small $1$};
		\node[xshift=-2*\vertsize]                       at (2) {\small $2$};		
		\node[xshift=2*\vertsize]                        at (3) {\small $3$};
		\node[yshift=-2*\vertsize]                       at (4) {\small $4$};
		\node[xshift=2*\vertsize]                        at (5) {\small $5$};		
		\node[xshift=-2*\vertsize]                       at (6) {\small $6$};
		\node[yshift=-2*\vertsize]                       at (7) {\small $7$};
		\node[yshift=-2*\vertsize,xshift=1.3*\vertsize]  at (8) {\small $8$};
    \end{scope}
\end{tikzpicture}
\end{center}

\noindent
Using the trivial $8$-coloring $\kappa$, the group $\Aut_\kappa(\Delta)$ is trivial,
and the $\N^8$-multigraded Hilbert series is
$$
\Hilb(\k[\Delta],\boldt)
=
1+\sum_{i=1}^8 \frac{t_i}{1-t_i}
+\sum_{
\substack{
ij \text{ in }\\ 
\left\{15,16,18,24,26,\right.
\\27,28,34,35,37,
\\ \left.38,48,58,,68\right\}
}}
\frac{t_i t_j}{(1-t_i)(1-t_j)}
+\sum_{
\substack{
ijk \text{ in }\\
\left\{158,168,248,\right.\\
\left.268,348,358\right\}
}}
\frac{t_i t_j t_k}{(1-t_i)(1-t_j)(1-t_k)}
$$
which specializes via $t_i=t$ to an $\N$-graded
Hilbert series in $\Z[[t]]:$
\begin{equation}
\label{N-graded-Hilb-example}
\begin{aligned}
\Hilb(\k[\Delta],t)
=1+\frac{8t}{1-t}+\frac{14t^2}{(1-t)^2}+
\frac{6t^3}{(1-t)^3} 
=\frac{1+5t+t^2-t^3}{(1-t)^3}
\end{aligned}
\end{equation}

On  the  other  hand,  $\Delta$  happens to have a  
proper  $3$-coloring  $\kappa: V \rightarrow [3]$:
$$
\begin{aligned}
1,2,3 & \mapsto 1,\\
4,5,6,7 & \mapsto 2,\\
8 & \mapsto 3.
\end{aligned}
$$
This $\kappa$ has one nontrivial color-preserving symmetry, $\sigma=(1)(4)(7)(8)(23)(56)$,
generating the two-element group $G=\Aut_\kappa(\Delta)=\{1,\sigma\}$.
Assuming that $\k$ does not have characteristic $2$, there
are exacty two simple $\k G$-modules, both one-dimensional: the trivial module $1$ and 
the nontrivial module in which $\sigma$ scales $\k$ by $-1$.  Denoting the class
of the nontrivial module by $\epsilon$, one can identify the
Grothendieck ring for $G$ as 
$
R_\k(G) \cong \Z[\epsilon]/(\epsilon^2-1).
$
One can then tabulate the $\kappa$-flag $f-$vector and $h-$vector entries, along with their equivariant generalizations,
as follows, using the fact that $G$-orbits of faces in $\Delta$ either have size one or two, and contribute either $1$ or $1+\epsilon$ to the equivariant $f-$vector entries:
\begin{center}
\begin{tabular}{|c|c|c|c|c|}\hline
$S$ & $f^\kappa_S$ &$h^\kappa_S$ &$[f^\kappa_S]$ &  $[h^\kappa_S]$ \\ \hline\hline
 $\varnothing$    &1  &1 &1 &1 \\\hline
 $\{1\}$    &$3$  &$2$ &$2+\epsilon$ &$1+\epsilon$ \\\hline
 $\{2\}$    &$4$  &$3$ &$3+\epsilon$ &$2+\epsilon$ \\\hline
 $\{3\}$    &$1$  &$0$ &$1$ &$0$ \\\hline
 $\{1,2\}$    &$8$  &$2$ &$4+4\epsilon$ &$2\epsilon$ \\\hline
 $\{1,3\}$    &$3$  &$0$ &$2+\epsilon$ &$0$ \\\hline
 $\{2,3\}$    &$3$  &$-1$ &$2+\epsilon$ &$-1$ \\\hline
 $\{1,2,3\}$    &$6$  &$-1$ &$3+3\epsilon$ &$-\epsilon$ \\\hline
\end{tabular}
\end{center}
For example, 
$[h^\kappa_{\{1,2\}}]=2\epsilon$ agrees with the subcomplex $\Delta_{\{1,2\}}$ being a graph with two independent $1$-cycles (or $1$-cocycles), both {\it reversing} orientation under the action of $\sigma$. On the other hand, $[h^\kappa_{\{2,3\}}]=-1$ because the subcomplex $\Delta_{\{2,3\}}$ is a graph with $\tilde{H}^1=0$ but $\tilde{H}^0=k$, where $\sigma$ fixes the $0$-cohomology class.

The $h^\kappa_S$ entries in the above table give,
via Proposition~\ref{Hilbert-series-calculations},
this $\N^3$-graded Hilbert series in $\Z[[t_1,t_2,t_3]]$:
$$
\Hilb(\k[\Delta],\boldt)
=\frac{1 + 2t_1 + 3t_2 + 2t_1 t_2- t_2 t_3 - t_1 t_2 t_3}{(1-t_1)(1-t_2)(1-t_3)}.
$$
Specializing $t_i=t$
again gives \eqref{N-graded-Hilb-example} above. The $[h^\kappa(S)]$ entries give 
this refinement in $R_\k(G)[[t_1,t_2,t_3]]$:
\begin{equation}
    \label{equivariant-multigraded-hilb-example}
\Hilb_{\equivariant}(\k[\Delta],\boldt)
=\frac{1 + (1+\epsilon) t_1 + (2+\epsilon) t_2 + 2\epsilon t_1 t_2- t_2 t_3 - \epsilon t_1 t_2 t_3}{(1-t_1)(1-t_2)(1-t_3)}.
\end{equation}
\end{example}

\subsection{Examples: Order complexes}

An important example of a balanced simplicial complex is the
order complex for a finite poset $P$, recalled here.

\begin{defn} \rm
Given a finite poset $P$, its {\it order complex} is
the simplicial complex $\Delta P$ with vertex set $V:=P$,
whose faces $F$ are the totally ordered subsets ({\it chains}) of $P$. 
\end{defn}

If the largest chain in $P$ has
$d$ elements, then $\Delta P$ has a proper $d$-coloring $V:=P \overset{\kappa}{\rightarrow} [d]$ defined by $\kappa(p)=\ell$ where $\ell$ is the number of elements in the
longest chain $p_1<p_2<\cdots<p_{\ell}:=p$ with top element $p$.  In this case, poset automorphisms of $P$
give rise to simplicial automorphisms of $\Delta P$,
and all such automorphisms respect this coloring $\kappa$,
so they lie in $\Aut_\kappa(\Delta P)$.

If the poset $P$ has all of its maximal chains of length $d$, then $\Delta P$ is a pure $(d-1)$-dimensional simplicial complex. 
The situation where $\Delta P$ is not only pure, but
also Cohen-Macaulay over $\k$ has been explored extensively since the work of Stanley \cite{Stanley-some-aspects}
on $[f^\kappa_S], [h^\kappa_S]$ in this setting\footnote{In \cite{Stanley-some-aspects}, representations of $\Aut_\kappa(\Delta)$ are over $\k={\mathbb C}$, and $[f^\kappa_S], [h^\kappa_S]$ are studied via their characters, called
$\alpha_S, \beta_S$ there.}.  In that situation,
because $\tilde{H}^i(\Delta|_S,\k)=0$
for $i \neq \# S-1$, the equivariant reduced Euler characteristic $\tilde{\chi}(\Delta|_S)$ has only one
nonvanishing term when computed as
in the right side of \eqref{Euler-Poincare}, 
simplifying the $\kappa$-flag $h$-vector:
\begin{equation}
\label{Cohen-Macaulay-poset-simplification}
[h^\kappa_S(\Delta)]=[\tilde{H}^{\#S-1}(\Delta|_S,\k)].
\end{equation}

In \cite{Stanley-some-aspects}, Stanley gave explicit irreducible decompositions for $[h^\kappa_S(\Delta)]$
in several interesting families of Cohen-Macaulay order
complexes $\Delta P$, some of which we discuss briefly here;  see \cite{Stanley-some-aspects}
for more details.

\begin{example} \rm
\label{Stanley-example-type-A}
The {\it Boolean algebra} $P=2^{[n]}$ is the poset of all subsets of $[n]$, ordered via inclusion.  Its order complex $\Delta P$ is Cohen-Macaulay over any field $\k$.  The symmetric group $S_n$ is the group of poset automorphisms of $P$, and hence a subgroup of $\Aut_\kappa(\Delta P)$.
When $\k$ has characteristic zero, the simple $\k[S_n]$-modules are indexed by {\it (number) partitions} $\lambda$ of $n$.
Denote by $[\lambda]$ the class  within $R_k(S_n)$ of the simple module
indexed by $\lambda$.  Recall that the
dimension of this simple module is the
number of {\it standard Young tableau $Q$ of shape $\lambda$}, which are labelings
of the cells of the boxes in the {\it Ferrers diagram} for $\lambda$ by the numbers $1,2,\ldots,|\lambda|=n$, increasing left-to-right in rows, and increasing top-to-bottom down columns.  For example, 
$$
Q=
\ytableausetup{smalltableaux}
\begin{ytableau} 1&2&4&7\\3&6\\5  
\end{ytableau} 
$$
is a standard Young tableau of {\it shape} $\lambda(Q):=(4,2,1)$.  One has a notion of
{\it descent set} for such a tableau:
$$
\Des(Q):=\{ i \in [n-1]:i+1 \text{ appears in a lower row than }i\text{ within }Q\}.
$$
For example, the tableau $Q$ shown above has $\Des(Q)=\{2,4\}$.

Stanley then proves the
following expression \cite[Thm. 4.3]{Stanley-some-aspects}
for the numerator on the far right side
of \eqref{equivariant-Hilb-expressions},
crediting it in different language to L. Solomon:
\begin{equation}
\label{Stanley-Solomon-type-A-formula}
\sum_{S \subset [n]}  [h^\kappa_S(\Delta P)] \cdot \boldt^S
=\sum_Q [\lambda(Q)] \cdot \boldt^{\Des(Q)}
\end{equation}
where here $Q$ runs over all {\it standard Young tableaux}
of size $n$.  This gives, via Proposition~\ref{Hilbert-series-calculations},
a very explicit expression for the $S_n$-equivariant Hilbert series of $\k[\Delta P]$:
\begin{equation}
\label{Stanley-Solomon-type-A-formula-hilb-consequence}
\Hilb_{\equivariant}(\k[\Delta P],\boldt)=
\frac{
\sum_Q [\lambda(Q)] \cdot \boldt^{\Des(Q)}
}
{\prod_{i=1}^n( 1-t_i )}
\end{equation}
\end{example}

\begin{example} \rm
\label{Stanley-example-type-B}
Stanley \cite[\S 6]{Stanley-some-aspects}
also proves a {\it type $B$} analogue
of the previous results.  He replaces the Boolean algebra with the poset of boundary faces of the $n$-dimensional {\it cross-polytope}, that is, the
convex hull of the vectors $\{ \pm e_1,\ldots,\pm e_n \}$
where $e_1,\ldots,e_n$ are standard basis vectors
in $\mathbb{R}^n$.  This face poset $P$ is isomorphic to a Cartesian
product $\{0,+1,-1\}^n$, with this
componentwise order:

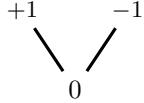
\begin{center}
\begin{tikzpicture}[scale=0.7]
\pgfmathsetlengthmacro{\rad}{25pt};
\pgfmathsetlengthmacro{\vertsize}{\rad*0.11};

\pgfmathsetmacro{\triangleopacity}{0.4};
\definecolor{dimtwocolor}{RGB}{140,140,200}
    \begin{scope}[shift={(-4.5,0)}]
		\node (0) at (0,-1.5) {$0$};
		\node (1)  at (-1,0) {\small $+1$};
		\node (2)  at (1,0)  {\small $-1$};	
		
		\draw [very thick, draw=black] (1) -- (0) -- (2);
    \end{scope}
    
\end{tikzpicture}
\end{center}

\noindent
The isomorphism sends an element $P=(\epsilon_1,\ldots,\epsilon_n)$ in $\{0,+1,-1\}^n$
to the boundary face of the cross-polytope
which is the convex hull of 
the vectors $\{\epsilon_i \cdot e_i:  \epsilon_i \neq 0\}$.

It is again true that 
$\Delta P$ is Cohen-Macaulay over any field $\k$.  The
group of poset automorphisms of $P$ is the
{\it hyperoctahedral group} $B_n$ of all
$n \times n$ signed permutation matrices, that is,
matrices in $\{0,\pm 1\}^{n \times n}$ having one nonzero entry in each row and column. Hence $B_n$ is a subgroup of the group $\Aut_\kappa(\Delta P)$.

When $\k$ has characteristic zero, the simple $\k[B_n]$-modules are indexed by {\it double partitions} of $n$,
which are ordered pairs $(\lambda^{(1)},\lambda^{(2)})$ of partitions whose 
sum of entries $|\lambda^{(1)}|+ |\lambda^{(2)}|=n$.
Denote by $[(\lambda^{(1)},\lambda^{(2)})]$ the class of this simple module within $R_k(B_n)$.
The dimension of this simple
module is given by the
number of {\it double standard Young tableaux} $Q=(Q_1,Q_2)$ of  {\it shape} $(\lambda^{(1)},\lambda^{(2)})$, where
each $Q_i$ is a labeling
of the cells of $\lambda^{(i)}$ with
values in $[n]$ so that
each $i$ in $[n]$ appears
exactly once, either in $Q_1$ or $Q_2$.

Stanley defines a notion of {\it descent set} 
for a double standard Young tableau $Q=(Q_1,Q_2)$:
$$
\begin{aligned}
\Des(Q):= 
&\{ i \in [n-1]: i,i+1 \text{ both appear in the same }Q_j, 
\text{ and }
i + 1\text{ appears in a lower row than }i\}\\
&\cup
\{i\in[n] : i\text{ appears in }Q_1,\text{ and }
i + 1\text{ in }Q_2,
\text{ or }
i = n\text{ and }n
\text{ appears in }Q_1\}.
\end{aligned}
$$
Repeating one of his examples, this double standard Young tableau 
$$
Q=(Q_1,Q_2)=
\left(
\ytableausetup{smalltableaux}
\begin{ytableau} 1&4&5\\6&9 
\end{ytableau},
\begin{ytableau} 2&7\\3\\8 
\end{ytableau}
\right)
$$
has $\Des(Q)=\{1,2,5,6,7,9\}$
and $(\lambda^{(1)}(Q),\lambda^{(2)}(Q))=((3,2),(2,1,1))$.

He then states and proves the following result \cite[Thm. 6.4]{Stanley-some-aspects} analogous to \eqref{Stanley-Solomon-type-A-formula}:
\begin{equation}
\label{Stanley-type-B-formula}
\sum_{S \subset [n]}  [h^\kappa_S(\Delta P)] \cdot \boldt^S
=\sum_Q [\lambda^{(1)}(Q),\lambda^{(2)}(Q)] \cdot \boldt^{\Des(Q)}
\end{equation}
where $Q$ runs over all double standard Young tableaux with $n$ cells.
Then Proposition~\ref{Hilbert-series-calculations}
again gives a very explicit expression for the $S_n$-equivariant Hilbert series of $\k[\Delta P]$:
\begin{equation}
\label{Stanley-type-B-formula-hilb-consequence}
\Hilb_{\equivariant}(\k[\Delta P],\boldt)=
\frac{
\sum_Q [\lambda^{(1)}(Q),\lambda^{(2)}(Q)] \cdot \boldt^{\Des(Q)}
}
{\prod_{i=1}^n( 1-t_i )}.
\end{equation}
\end{example}

Our last family of Cohen-Macaulay order complexes $\Delta P$ were
studied by Athanasiadis \cite{Athanasiadis}, who decomposed
$[h^\kappa_S(\Delta P)]$ into irreducibles.  In fact, this
family provided the original motivation for our study.

\begin{example} \rm
\label{Athanasiadis-example}
The poset $P$ of {\it injective words on $n$ letters} has as  its underlying set all words in the alphabet $[n]$ using each letter at most once.
One has $u \leq v$ in $P$ if $u$ is a (not necessarily contiguous) {\it subword} of $v=(v_1,\ldots,v_m)$, meaning that $u=(v_{i_1},\ldots,v_{i_\ell})$ for
some indices $1 \leq i_1 < i_2 < \cdots < i_\ell \leq m$.  We depict $P$ here for $n=2,3$, abbrevating a word 
$v=(v_1,v_2,\ldots,v_m)$ as $v_1 v_2 \cdots v_m$:

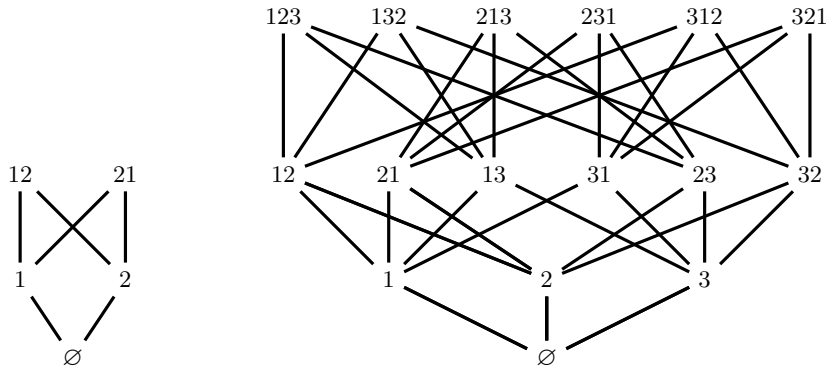
\begin{center}
\begin{tikzpicture}[scale=0.7]
\pgfmathsetlengthmacro{\rad}{25pt};
\pgfmathsetlengthmacro{\vertsize}{\rad*0.11};

\pgfmathsetmacro{\triangleopacity}{0.4};
\definecolor{dimtwocolor}{RGB}{140,140,200}

    \begin{scope}[shift={(-4.5,0)}]
    
		\node (emptyset) at (0,-1.5) {$\varnothing$};
		
		\node (1)  at (-1,0) {\small $1$};
		\node (2)  at (1,0)  {\small $2$};	
		\node (12) at (-1,2) {\small $12$};		
		\node (21) at (1,2)  {\small $21$};		
		
		\draw [very thick, draw=black] (emptyset) -- (1) -- (12);
		\draw [very thick, draw=black] (1) -- (21);
		\draw [very thick, draw=black] (2) -- (12);
	    \draw [very thick, draw=black] (emptyset) -- (2) -- (21);
    \end{scope}

    \begin{scope}[shift={(4.5,0)}]
		\node (emptyset) at (0,-1.5) {$\varnothing$};
		
		\node (1)  at (-3,0) {\small $1$};
		\node (2)  at (0,0)  {\small $2$};
		\node (3)  at (3,0)  {\small $3$};	
		\node (12) at (-5,2) {\small $12$};		
		\node (21) at (-3,2)  {\small $21$};
		\node (13) at (-1,2) {\small $13$};		
		\node (31) at (1,2)  {\small $31$};
		\node (23) at (3,2) {\small $23$};		
		\node (32) at (5,2)  {\small $32$};
		\node (123) at (-5,5)  {\small $123$};
		\node (132) at (-3,5)  {\small $132$};
		\node (213) at (-1,5)  {\small $213$};
		\node (231) at (1,5)  {\small $231$};
		\node (312) at (3,5)  {\small $312$};
		\node (321) at (5,5)  {\small $321$};
		
		\draw [very thick, draw=black] (emptyset) -- (1) -- (12);
		\draw [very thick, draw=black] (1) -- (21);
		\draw [very thick, draw=black] (2) -- (12);
	    \draw [very thick, draw=black] (emptyset) -- (2) -- (21);
	    \draw [very thick, draw=black] (emptyset) -- (1) -- (13);
		\draw [very thick, draw=black] (1) -- (31);
		\draw [very thick, draw=black] (2) -- (23);
	    \draw [very thick, draw=black] (emptyset) -- (2) -- (32);
	    \draw [very thick, draw=black] (emptyset) -- (3) -- (23);
		\draw [very thick, draw=black] (3) -- (32);
		\draw [very thick, draw=black] (2) -- (12);
	    \draw [very thick, draw=black] (emptyset) -- (2) -- (21);
	    \draw [very thick, draw=black] (emptyset) -- (3) -- (31);
		\draw [very thick, draw=black] (3) -- (13);
		\draw [very thick, draw=black] (12) -- (123);
		\draw [very thick, draw=black] (12) -- (132);
		\draw [very thick, draw=black] (12) -- (312);
		\draw [very thick, draw=black] (13) -- (123);
		\draw [very thick, draw=black] (13) -- (132);
		\draw [very thick, draw=black] (13) -- (213);
		\draw [very thick, draw=black] (23) -- (123);
		\draw [very thick, draw=black] (23) -- (213);
		\draw [very thick, draw=black] (23) -- (231);
		\draw [very thick, draw=black] (21) -- (213);
		\draw [very thick, draw=black] (21) -- (231);
		\draw [very thick, draw=black] (21) -- (321);
		\draw [very thick, draw=black] (31) -- (321);
		\draw [very thick, draw=black] (31) -- (312);
		\draw [very thick, draw=black] (31) -- (231);
		\draw [very thick, draw=black] (32) -- (132);
		\draw [very thick, draw=black] (32) -- (312);
		\draw [very thick, draw=black] (32) -- (321);
    \end{scope}
    
\end{tikzpicture}
\end{center}

The symmetric group $S_n$ permutes the letters $[n]$, and thus permutes the 
injective words $u=(u_1,\ldots,u_\ell)$ via 
$w(u):=(w(u_1),\ldots,w(u_\ell))$.  Hence $S_n$
is a subgroup of poset automorphisms of $P$,
and of $\Aut_\kappa(\Delta P)$. 

It is known that $\Delta P$ is Cohen-Macaulay over any field $\k$.
We review here Athanasiadis's description of 
$[h^\kappa_S(\Delta P)]$;  see \cite{Athanasiadis} for more details.
When $\k$ has characteristic zero, we will use the same notation
$[\lambda]$ for the class of the irreducible $\k S_n$-module indexed by $\lambda$ as in Example~\ref{Stanley-example-type-A}.
For permutations
$w=(w_1,\ldots,w_n)$ in $S_n$
introduce their usual {\it descent set} 
$$
\Des(w):=\{i \in [n-1]: w_i > w_{i+1}\}.
$$
For each pair $(w,Q)$ of a permutation $w$ in $S_n$ and standard Young tableau $Q$ of size $n$, introduce a certain statistic $\tau(w,Q)$ taking
values in $\{0,1,2\ldots,n\}$, defined as follows.
If $\Des(w)=S=\{s_1 < s_2 < \cdots <s_k\}$, with
convention $s_0=0, s_{k+1}=n$, let
$w_S$ be the unique longest permutation in $S_n$ (the one with most
{\it inversions} $i<j$ with $w(i) > w(j)$) satisfying $\Des(w_S)=S=\Des(w)$.
Then define 
$\tau(w,Q)$ be the largest index $i$ in  
$\{0, 1, \ldots , k + 1\}$ for which
both $w(x) = w_S(x)$ for all $x > s_{k-i+1}$ and
 $\min \Des(Q) \geq n-s_{k-i+1}$.

Athanasiadis then proves 
this expression \cite[Thm. 1.2]{Athanasiadis} for the numerator on the far right 
of \eqref{equivariant-Hilb-expressions}:
\begin{equation}
\label{Athanasiadis-formula}
\sum_{S \subset [n]}  [h^\kappa_S(\Delta P)] \cdot \boldt^S
=\sum_{Q} [\lambda(Q)] 
\left(
\sum_{\substack{w\in S_n:\\\tau(w,Q)\text{ odd}}} \boldt^{\Des(w)}  +
 t_n \sum_{\substack{w\in S_n:\\\tau(w,Q)\text{ even}}} \boldt^{\Des(w)}
\right).
\end{equation}
Here $Q$ runs over all standard Young tableaux with $n$ cells.
Again Proposition~\ref{Hilbert-series-calculations} gives
an expression for the $S_n$-equivariant Hilbert series of $\k[\Delta P]$, with numerator \eqref{Athanasiadis-formula} and denominator
$\prod_{i=1}^n(1-t_i)$
\end{example}

\subsection{Equivariant resolutions and Tor}
\label{equivariant-resolution-section}

One way to compute the Hilbert series of a finitely generated graded module $M$ over a graded ring $A$ is by an $A$-free resolution of $M$.  This still holds in an equivariant setting where one has a finite group $G$ acting on $M$
in a grade-preserving fashion, but one must be slightly
more carefully about the statements.  We collect here
some of the facts that we will need in our setting.

We will work with $A=\k[z_1,\ldots,z_d]$ a polynomial ring, possibly multigraded, and $M$ a finitely generated multigraded $A$-module.  Assume one is given a finite
group $G$ that acts trivially on $A$, that is, fixing it
pointwise.  Also assume that $G$ acts on $M$ in a grade-preserving fashion
that commutes with the $A$-module structure, that is,
$g(am) = ag(m)$ for all $a$ in $A$ and $m$ in $M$.

\begin{prop}
\label{equivariant-resolution-prop}
In the above setting, there exists an equivariant finite free $A$-resolution $\mathcal{F}$ of $M$
\begin{equation}
    \label{generic-MFR}
\mathcal{F}: 0 \rightarrow F_d \rightarrow \cdots \rightarrow F_0 \rightarrow M \rightarrow 0.
\end{equation}
Here each $F_i$ is both a free $A$-module
of finite rank 
and a $\k G$-module, of the form
$A \otimes_\k U_i$ for some finite-dimensional graded $\k G$-module $U_i$, with all maps being $A$-module and $\k G$-module
morphisms.

This gives an expression for
the equviariant Hilbert series of $M$ as
\begin{equation}
\label{general-resolution-hilb-expression}
\begin{aligned}
\Hilb_\equivariant(M,\boldt)
&=\Hilb(A,\boldt) \sum_{i=0}^d (-1)^i \Hilb_\equivariant(U_i,\boldt) \\
&= \Hilb(A,\boldt) \sum_{i=0}^d (-1)^i \Hilb_\equivariant(\Tor^A_i(M,\k) ,\boldt).
\end{aligned}
\end{equation}
This resolution $\mathcal{F}$ is not necessarily minimal,
but when
$\k G$ is semisimple (so $\# G$ lies in $\k^\times$), then it may be chosen minimally.
In this case, one has $\k G$-module isomorphisms $\Tor_i^A(M,\k) \cong U_i$ for $i=0,1,2,\ldots,d$.
\end{prop}
\begin{proof}
This is
\cite[Prop. 2.1(i)-(iv)]{BRSW} for polynomial rings with trivial $G$-action, and modules over them.
\end{proof}

\section{A colorful Hochster formula}
\label{colorful-Hochster-section}

Having fixed a $d$-coloring $\kappa$ of $\Delta$, 
the denominator $\prod_{j=1}^d (1-t_j)$ on the rightmost side of \eqref{equivariant-Hilb-expressions} suggests regarding
$\k[\Delta]$ as a module over an auxiliary
polynomial ring $A:=\k[z_1,\ldots,z_d]$, with an $\N^d$-multigrading
in which $\deg(z_j)=\epsilon_j$.  One can naturally endow $\k[\Delta]$ with such 
an $A$-module structure if one lets $z_j$ act on $\k[\Delta]$ as multiplication by
the following element $\gamma_j$ in $\k[\Delta]$, the sum of all vertices of color $j$:
$$
\gamma_j:=\sum_{\substack{i \in [n]:\\ \kappa(i)=j}} x_j.
$$

The following proposition shows how
 $z_j$ acts on the monomial $\k$-basis 
for $\k[\Delta]$.  We omit the proof, which is
straightforward, using the properness of the $d$-coloring $\kappa$.

\begin{prop}
\label{monomial-calculation-prop}
Given $\bolda \in \N^n$ with $\supp(\bolda) \in \Delta$,
then $z_j$ acts on $\boldx^\bolda=x_1^{a_1}\cdots x_n^{a_n}$
 in $\k[\Delta]$ as follows:
\begin{itemize}
    \item[(i)] If $j$ appears in $\kappa(\supp(\bolda))$, say
    $a_i >0$ and $\kappa(i)=j$, then
$$
z_j \left( \boldx^{\bolda} \right)
:=\gamma_j \cdot \boldx^{\bolda}
=\boldx^{\bolda+e_i}.
$$
\item[(ii)] If $j$ does not appear in $\kappa(\supp(\bolda))$,
then 
$$
z_j \left( \boldx^{\bolda} \right)
:=\gamma_j \cdot \boldx^{\bolda}
=
   \displaystyle\sum_{i} \boldx^{\bolda + e_i} 
$$
where the sum is over vertices $i$ in $V=[n]$ for which $\kappa(i)=j$ and
$\supp(\bolda) \cup \{i\}$ is a face in $\Delta$. 
\end{itemize}
\end{prop}

\begin{cor}
In the above setting, 
$\k[\Delta]$ is finitely generated over
$A=\k[z_1,\ldots,z_d]$, by $\{ \boldx^F: F \in \Delta\}$. 
\end{cor}
\begin{proof}
Proposition~\ref{monomial-calculation-prop} case (i)
shows that if a face $F=\{i_1,\ldots,i_r\}$ of $\Delta$
has vertices colored $\kappa(i_\ell)=j_\ell$ for $\ell=1,2,\ldots,r$, then
the $\k$-basis element
$
\prod_{i \in F} x_i^{a_i}
$
with $a_i \geq 1$ can be re-written as
$
z_{j_1}^{a_1-1}\cdots z_{j_r}^{a_r-1} \cdot  \boldx^F.
$
\end{proof}

Note that when $\Aut_\kappa(\Delta)$ acts on
$\k[\Delta]$, it fixes each of $\gamma_1,\ldots,\gamma_d$.
Therefore, regarding $A$ as having trivial $\Aut_\kappa(\Delta)$-action, Proposition~\ref{equivariant-resolution-prop} 
applies to the $\N^d$-graded polynomial ring $A$
and $\N^d$-graded $A$-module $\k[\Delta]$.
One can therefore consider $\Tor_m^A(\k[\Delta],\k)$
as an $\N^d$-graded $k$-vector space,
whose $\boldb$-homogeneous component
will be denoted $\Tor_m^A(\k[\Delta],\k)_\boldb$.
Our colorful version of Hochster's Formula \cite{Hochster}
expresses  $\Tor_m^A(\k[\Delta],\k)_\boldb$ in terms of the (reduced) cohomologies $\tilde{H}^*(\Delta|_S,\k)$,
where for $S \subseteq [d]$, 
the {\it color-selected subcomplex} $\Delta|_S$
is defined in \eqref{color-selected-subcomplex-defn}.
Note that $\Aut_\kappa(\Delta)$ acts as automorphisms on
each $\Delta|_S$, and on $\tilde{H}^*(\Delta|_S)$.

\begin{thm}
\label{colorful-Hochster-formula}
{\bf (Colorful Hochster formula)}
Fix any proper $d$-coloring $\kappa$ of a simplicial complex $\Delta$.  

Then in the above notations, for any $\boldb$
in $\N^d$, one has
$$
\Tor^A_m(\k[\Delta],\k)_\boldb
\cong
\begin{cases}
0 & \text{ if }
\boldb \not\in \{0,1\}^d,\\
\tilde{H}^{\#S-m-1}(\Delta|_S,\k) 
& \text{ if }
\boldb =\sum_{j \in S} \epsilon_j \in \{0,1\}^d.
\end{cases}
$$
Furthermore, these $\k$-vector space isomorphisms are equivariant with respect to
the group $\Aut_\kappa(\Delta)$.
\end{thm}

\begin{remark} \rm
If $\kappa$ is the trivial $n$-coloring of $V=[n]$,
Theorem~\ref{colorful-Hochster-formula} is Hochster's Formula 
\cite[Thm. 5.1]{Hochster}. 
For an interesting generalization of Hochster's formula in a different direction, see Bruns, Koch and R\"omer \cite[\S4]{BrunsKochRomer}. If
$\kappa$ is a balanced $d$-coloring for $\Delta$, Theorem~\ref{colorful-Hochster-formula} is closely related
to Conjecture~\ref{canonical-resolution} below.
\end{remark}

\begin{remark} \rm
We note that Theorem~\ref{colorful-Hochster-formula}
gives a second proof of the rightmost expression in \eqref{equivariant-Hilb-expressions}
for the equivariant Hilbert series of 
$\k[\Delta]$:
Applying Proposition~\ref{equivariant-resolution-prop},
one has
$$
\begin{aligned}
\Hilb_{\equivariant}(\k[\Delta],\boldt)
&=\Hilb(A,\boldt) \sum_{m=0}^d (-1)^m \,\,
 \Hilb_\equivariant \left(\Tor^A_m(\k[\Delta],\k),\boldt\right)\\
&=\frac{1}{\prod_{j=1}^d (1-t_j)}
     \sum_{m=0}^d (-1)^m \,\, 
    \sum_{S \subseteq [d]} [\tilde{H}^{\#S-m-1}(\Delta|_S,\k)] \cdot \boldt^S\\
&=\frac{1}{\prod_{j=1}^d (1-t_j)}
   \sum_{S \subseteq [d]} [h^\kappa_S(\Delta)] \cdot \boldt^S
\end{aligned}
$$
where the first equality used 
Proposition~\ref{equivariant-resolution-prop},
the second used Theorem~\ref{colorful-Hochster-formula}
and the third applied the definitions
\eqref{equivariant-h-definition}, \eqref{Euler-Poincare}
of $[h^\kappa_S(\Delta)]$.
\end{remark}

\begin{example} 
\label{example-revisited}
\rm
Continuing with the simplicial complex $\Delta$ from Example~\ref{running-example-one},
using the trivial $8$-coloring $\kappa$, one obtains the resolution whose shape is predicted by Hochster's original formula.  It has homological dimension $6=8-2$, as predicted by the Auslander-Buchsbaum Theorem \cite[Thm. 19.9]{Eisenbud}, since the depth of $\k[\Delta]$ is $2$.  Here is some (singly-graded) {\tt Macaulay2} output:

\begin{verbatim}
i1 : S = QQ[x_1..x_8];

i2 : IDelta = ideal(x_1*x_2, x_1*x_3, x_1*x_4, x_1*x_7, x_2*x_3, x_2*x_5, x_3*x_6,
                    x_4*x_5, x_4*x_6, x_4*x_7, x_5*x_6, x_5*x_7, x_6*x_7, x_7*x_8);

i3 : betti res IDelta;

            0  1  2  3  4 5 6
o3 = total: 1 14 36 39 22 7 1
         0: 1  .  .  .  . . .
         1: . 14 34 32 11 1 .
         2: .  .  2  7 11 6 1
\end{verbatim}

\noindent
For example, here the southeasternmost $1$ entry in the Betti table comes from the fact that $\tilde{H}^1(\Delta)=\k^1$, while the entry of $6=1+1+1+1+2$
directly to its left comes from 
$$
\tilde{H}^1(
\Delta|_{\{1,2,3,4,5,6,7,8\}\setminus\{i\}}
)=
\begin{cases}
0 & \text{ if } i=2,3,7,\\
\k^1 & \text{ if }i=1,4,5,6,\\
\k^2 & \text{ if }i=8.
\end{cases}
$$
On the other hand, using the proper $3$-coloring $\kappa$ of $\Delta$ discussed in the same example, one obtains a much shorter resolution of $\k[\Delta]$ over $A=\k[z_1,z_2,z_3]$, having homological dimension $1=3-2$, as shown here:
\newpage
\begin{verbatim}
i4 : phi = map(S, QQ[z_1..z_3], matrix {{x_1+x_2+x_3, x_4+x_5+x_6+x_7, x_8}});

i5 : betti res pushForward(phi, S^1/IDelta);

            0 1
o5 = total: 8 2
         0: 1 .
         1: 5 1
         2: 2 1
\end{verbatim}
The equivariant and $\N^3$-multigraded refinement 
of the above $\N$-graded nonequvariant Betti table is this:
\begin{equation}
\label{equivariant-multigraded-Betti-table-example}
\begin{tabular}{|c||c|}\hline
$
\Hilb_{\equivariant}( \Tor_0^A(\k[\Delta],\k), \boldt)
$
& 
$
\Hilb_{\equivariant}( \Tor_1^A(\k[\Delta],\k), \boldt)
$\\\hline\hline
$1$&  \\\hline
$+(1+\epsilon)t_1+(2+\epsilon)t_2$ & 
$1 \cdot t_2 t_3$ \\\hline
$+2\epsilon \cdot t_1 t_2$ & 
$+\epsilon \cdot t_1 t_2 t_3$ \\\hline
\end{tabular}
\end{equation}
which one can check is consistent with the equivariant $\N^3$-graded
Hilbert series shown in \eqref{equivariant-multigraded-hilb-example}.

\end{example}

Our proof of Theorem~\ref{colorful-Hochster-formula}
simply generalizes Hochster's proof of
his original formula \cite[Thm. 5.1]{Hochster}.

\begin{proof}[Proof of Theorem~\ref{colorful-Hochster-formula}.]
We compute $\Tor^A(\k[\Delta],\k)$ via a Koszul resolution $\K$ of $\k$.  Here $\k$ is the trivial $A$-module $\k=A/(z_1,\ldots,z_m)$, carrying trivial action of $\Aut_\kappa(\Delta)$.  This Koszul resolution $\K$ has
$m^{th}$ term 
$$
K_m=A \otimes_\k \wedge^m \k^d
$$
where $\k^d$ has standard basis elements $\epsilon_1,\ldots,\epsilon_d$.
Applying $\k[\Delta] \otimes_A(-)$ gives a complex $\k[\Delta] \otimes_A \K$ of $A$-modules,
whose homology computes $\Tor^A(\k[\Delta],\k)$.
The $m^{th}$ term of $\k[\Delta] \otimes_A \K$ is the $A$-module
$$
\k[\Delta] \otimes_A A \otimes_\k \wedge^m \k^d
\quad \cong \quad 
\k[\Delta]\otimes_\k \wedge^m \k^d,
$$
where $A=\k[z_1,\ldots,z_d]$ acts with
$z_j$ multiplying by $\gamma_j$ in the left tensor factor of 
$\k[\Delta]\otimes_\k \wedge^m \k^d$. The group $\Aut_\kappa(\Delta)$ also acts trivially on the right tensor factor $\wedge^m \k^d$, but nontrivially on the left factor $\k[\Delta]$.  The differential $\partial$ acts on a $\k$-basis element 
$
\boldx^\bolda \otimes \epsilon_{j_1} \wedge \cdots \wedge \epsilon_{j_m}
$
with $1 \leq j_1 < \cdots < j_m \leq d$ as follows:
\begin{equation}
\label{Koszul-differential}
\partial\left( 
\boldx^\bolda \otimes \epsilon_{j_1} \wedge \cdots \wedge \epsilon_{j_m}
\right)
=\sum_{\ell=1}^m (-1)^{\ell-1}
  \gamma_{j_\ell} \cdot \boldx^\bolda \otimes 
  \epsilon_{j_1} \wedge \cdots \wedge \widehat{\epsilon_{j_\ell}} \wedge \cdots \wedge \epsilon_{j_m}
\end{equation}
Given $\boldb$ in $\N^d$, we consider the $\boldb$-multigraded strand $\left(\k[\Delta] \otimes_A \K\right)_\boldb$, whose $\k$-basis
are the elements 
\begin{equation}
    \label{strand-basis-elements}
\{
\boldx^\bolda \otimes \epsilon_{j_1} \wedge \cdots \wedge \epsilon_{j_m}:
\supp(\bolda) \in \Delta \text{ and }
1 \leq j_1 < \cdots j_m \leq d \text{ and }
\deg_{\N^d}(\boldx^\bolda)+\epsilon_{j_1}+ \cdots+\epsilon_{j_m}=\boldb 
\}.
\end{equation}
We show $\left(\k[\Delta] \otimes_A \K\right)_\boldb$ is acyclic if
$\boldb \not\in \{0,1\}^d$, and otherwise identify it
with $\tilde{C}(\Delta|_S,\k)$ if $\boldb =\sum_{j \in S} \epsilon_j$.

\vskip.1in
\noindent
{\sf Case 1.} The multidegree $\boldb$ does not lie in $\{0,1\}^d$.

Here we wish to show $\left(\k[\Delta] \otimes_A \K\right)_\boldb$
is acyclic.
Since $\boldb$ lies in $\N^d$ but not in $\{0,1\}^d$,
we may assume without loss of generality, by reindexing the coordinates,
that it has first coordinate $b_1 \geq 2$.  
We will use this to define a $\k$-linear {\it chain contraction}  
$$
\left(\k[\Delta] \otimes_A \K_m \right)_\boldb
\overset{D}{\longrightarrow}
\left(\k[\Delta] \otimes_A \K_{m+1}\right)_\boldb
$$
satisfying $\partial D+D \partial=1$, which then implies acyclicity.

To define $D$, note that the inequality $b_1 \geq 2$ together with the conditions in \eqref{strand-basis-elements} imply that $\boldx^\bolda$ is divisible by at
least one variable $x_{i_0}$ with color $\kappa(i_0)=1$.
But then the fact that $\kappa$ is a proper vertex-coloring, 
along with the condition in \eqref{strand-basis-elements} 
that $\supp(\bolda)$ lies in $\Delta$, forces this variable
$x_{i_0}$ to be unique.  Thus we can define
$$
D\left( 
\boldx^\bolda \otimes \epsilon_{j_1} \wedge \cdots \wedge \epsilon_{j_m} \right)
:= \boldx^{\bolda-e_{i_0}} \otimes \epsilon_1 \wedge \epsilon_{j_1} \wedge \cdots \wedge \epsilon_{j_m}
$$
and extend this map $\k$-linearly to all of 
$\left(\k[\Delta] \otimes_A \K_m \right)_\boldb$.
It remains to check that $\partial D+D \partial$ acts as the identity 
on each $\k$-basis element from \eqref{strand-basis-elements}.  There are two cases to consider, namely  $j_1 \geq 2$ or $j_1=1$.

If $j_1 \geq 2$, then we calculate
$$
\begin{aligned}
\partial D
\left( \boldx^\bolda \otimes \epsilon_{j_1} \wedge \cdots \wedge \epsilon_{j_m} \right)
&=\partial \left(
  \boldx^{\bolda-e_{i_0}} \otimes \epsilon_1 \wedge \epsilon_{j_1} \wedge \cdots \wedge \epsilon_{j_m}
\right)\\
&=\gamma_1 \boldx^{\bolda-e_{i_0}} \otimes \epsilon_{j_1} \wedge \cdots \wedge \epsilon_{j_m}\\
&\qquad -\sum_{\ell=1}^m (-1)^{\ell-1} \gamma_{j_\ell}  \boldx^{\bolda-e_{i_0}} \otimes \epsilon_1 \wedge \epsilon_{j_1} \wedge \cdots  \wedge \widehat{\epsilon_{j_\ell}} \wedge \cdots \wedge \epsilon_{j_m} \\
\end{aligned}
$$
and also calculate
$$
\begin{aligned}
D \partial 
\left( \boldx^\bolda \otimes \epsilon_{j_1} \wedge \cdots \wedge \epsilon_{j_m} \right)
&= \sum_{\ell=1}^m (-1)^{\ell-1}  D\left( \gamma_{j_\ell} \boldx^{\bolda} \otimes \epsilon_{j_1} \wedge \cdots  \wedge \widehat{\epsilon_{j_\ell}} \wedge \cdots \wedge \epsilon_{j_m} \right) \\
&= \sum_{\ell=1}^m (-1)^{\ell-1}  \gamma_{j_\ell} \boldx^{\bolda-\epsilon_{i_0}} \otimes \epsilon_1 \wedge \epsilon_{j_1} \wedge \cdots  \wedge \widehat{\epsilon_{j_\ell}} \wedge \cdots \wedge \epsilon_{j_m} .
\end{aligned}
$$
Adding the previous two calculations shows that
$$
\left(\partial D+D\partial \right)
\left( \boldx^\bolda \otimes \epsilon_{j_1} \wedge \cdots \wedge \epsilon_{j_m} \right)
= \gamma_1 \boldx^{\bolda-e_{i_0}} \otimes \epsilon_{j_1} \wedge \cdots \wedge \epsilon_{j_m}
= \boldx^{\bolda} \otimes \epsilon_{j_1} \wedge \cdots \wedge \epsilon_{j_m}
$$
where the last equality used the fact that $i_0$ is the unique vertex with $\kappa(i_0)=1$ in $\supp(\bolda)$, and $a_1 \geq 2$, allowing us to employ case (i) from
Proposition~\ref{monomial-calculation-prop}.

If $j_1=1$, then the fact that $\epsilon_1 \wedge \epsilon_1=0$ implies that $D$ annihilates 
$\boldx^\bolda \otimes \epsilon_1 \wedge \epsilon_{j_2} \wedge \cdots \wedge \epsilon_{j_m}$,
and hence so does $\partial D$.  
On the other hand,
$$
\begin{aligned}
&D \partial 
\left( \boldx^\bolda \otimes \epsilon_1 \wedge \epsilon_{j_2} \wedge \cdots \wedge \epsilon_{j_m} \right)\\
&= D\left(
\gamma_1 \boldx^{\bolda} \otimes \epsilon_{j_2} \wedge \cdots \wedge \epsilon_{j_m}
\right)
+
 \sum_{\ell=2}^m (-1)^{\ell-1} D\left( \gamma_{j_\ell} \boldx^\bolda \otimes \epsilon_1 \wedge \epsilon_{j_2} \wedge \cdots  \wedge \widehat{\epsilon_{j_\ell}} \wedge \cdots \wedge \epsilon_{j_m} 
\right)\\
&= D\left(
 \boldx^{\bolda+e_{i_0}} \otimes \epsilon_{j_2} \wedge \cdots \wedge \epsilon_{j_m}
\right)
=\boldx^{\bolda} \otimes \epsilon_1 \wedge \epsilon_{j_2} \wedge \cdots \wedge \epsilon_{j_m}
\end{aligned}
$$
where in the first line, the terms in the summation on $\ell$ all vanish again because $\epsilon_1 \wedge \epsilon_1=0$.

Thus we have checked $\partial D + D \partial$ fixes
each basis element $\left(\k[\Delta] \otimes_A \K\right)_\boldb$, showing acyclicity.

\vskip.1in
\noindent
{\sf Case 2.} The multidegree $\boldb$ lies  in $\{0,1\}^d$, so
$\boldb =\sum_{j \in S} \epsilon_j$ for some $S \subseteq [d].$

We wish to identify 
$
\tilde{H}^{\#S-1-m}(\Delta|_S,\k) 
\cong
\Tor^A_m(\k[\Delta],\k)_\boldb,
$
by exhibiting a chain complex isomorphism
\begin{equation}
\label{chain-complex-isomorphism}
\tilde{C}(\Delta|_S,\k) \overset{\varphi}{\longrightarrow}
\left(\k[\Delta] \otimes_A \K \right)_\boldb
\end{equation}
where $\tilde{C}(\Delta|_S,\k)$ is an augmented simplicial cochain
complex computing (reduced) cohomology $\tilde{H}(\Delta|_S,\k)$.

We recall one way to set up this complex, by 
first fixing a total order $\prec$ on the color set $S$.
This allows one to define a {\it sign} 
$\sgn(s_1,\ldots,s_p) \in \{\pm 1\}$ 
for any ordered $p$-subset of $S$, as
the sign of the permutation that
sorts $(s_1,\ldots,s_p)$ into its $\prec$-order.  Then since each $(p-1)$-dimensional
face $F=\{i_1,\ldots,i_p\}$ in $\Delta$ has at most one vertex of each color in $S$, one can reindex so that
$\kappa(i_1) \prec \cdots \prec \kappa(i_p)$, and
choose a $\k$-basis element $[i_1,\ldots,i_p]^*$ 
within the oriented cochains $\tilde{C}^{p-1}(\Delta|_S,\k)$
corresponding to the face $F$;  these cochains $\{[i_1,\ldots,i_p]^*\}$ are the dual basis elements to the oriented simplices $\{ [i_1,\ldots,i_p] \}$ that form a basis for the simplicial chains $\tilde{C}_{p-1}(\Delta|_S,\k)$.  To express the simplicial coboundary map
$
\tilde{C}^{p-1}(\Delta|_S,\k) \overset{\delta}{\longrightarrow}
\tilde{C}^p(\Delta|_S,\k),
$
adopt the sign orientation convention that 
$$
[i_{\sigma_1},\ldots,i_{\sigma_p}]^*
=\sgn(\sigma) \cdot  [i_1,\ldots,i_p]^*
$$
for any permutation $\sigma \in \mathfrak{S}_p$,
and then the  coboundary map looks as follows:
$$
\delta[i_1,\ldots,i_p]^*
=\sum_{\substack{i \in [n]:\\\kappa(i) \in S, \\ F \cup \{i\} \in \Delta}} 
\sgn( \kappa(i_1), \ldots, \kappa(i_p), \kappa(i)) \cdot
[i_1,\ldots,i_p,i]^*.
$$
If $S \setminus F=\{j_1,\ldots,j_m\}$ with $j_1 \prec \cdots \prec j_m$ in the ordering on $S$,
then this can be re-expressed as
\begin{equation}
    \label{re-expressed-simplicial-coboundary-map}
\delta[i_1,\ldots,i_p]^*
:=\sum_{\ell=1}^m \quad
 \sum_{\substack{i \in [n]:\\ \kappa(i)=j_\ell, \\ F \cup \{i\} \in \Delta}} 
\sgn( \kappa(i_1), \ldots, \kappa(i_p), j_\ell) \cdot
[i_1,\ldots,i_p,i]^*.
\end{equation}
Having fixed these notations, one can define the isomorphism $\varphi$ from \eqref{chain-complex-isomorphism}
by mapping the basis as follows: 
\begin{equation}
\label{definition-of-cochain-isomorphism}
[i_1,\ldots,i_p]^* \quad
\overset{\varphi}{\longmapsto} \quad
\sgn( \kappa(i_1),\ldots,\kappa(i_p), j_1, \ldots, j_m) \cdot
x_{i_1} \cdots x_{i_p} \otimes \epsilon_{j_1} \wedge \cdots\wedge \epsilon_{j_m}
\end{equation}
Note the correspondence in homological degrees here:  the basis element on the left lies in $\tilde{C}^{p-1}(\Delta|_S,\k)$, and
maps to an element of $\k[\Delta] \otimes_A \K_m$, where $p=\#S-m$.
It is not hard to check from the conditions in \eqref{strand-basis-elements}
on the typical basis element $\boldx^\bolda \otimes \epsilon_{j_1} \wedge \cdots \wedge \epsilon_{j_m}$
that this map sends our chosen basis of $\tilde{C}(\Delta|_S,\k)$
to the basis of $\left(\k[\Delta] \otimes_A \K \right)_\boldb$,
so it is a $\k$-vector space isomorphism.
To check that it is an isomorphism of complexes, note the action of the differential
\eqref{Koszul-differential} on 
$x_{i_1} \cdots x_{i_p} \otimes \epsilon_{j_1} \wedge \cdots \wedge \epsilon_{j_m}$
is as follows:
$$
\begin{aligned}
\partial\left(
 x_{i_1} \cdots x_{i_p} \otimes \epsilon_{j_1} \wedge \cdots \wedge \epsilon_{j_m}
\right)
& = \sum_{\ell=1}^m (-1)^{\ell-1} \gamma_{j_\ell}  x_{i_1} \cdots x_{i_p} \otimes 
       \epsilon_{j_1} \wedge \cdots \wedge \widehat{\epsilon_{j_\ell}} \wedge \cdots \wedge \epsilon_{j_m} \\
& = \sum_{\ell=1}^m (-1)^{\ell-1} 
 \sum_{\substack{i \in [n]:\\ \kappa(i)=j_\ell, \\ F \cup \{i\} \in \Delta}} 
x_{i_1} \cdots x_{i_p} x_i \otimes 
       \epsilon_{j_1} \wedge \cdots \wedge \widehat{\epsilon_{j_\ell}} \wedge \cdots \wedge \epsilon_{j_m}
\end{aligned}
$$
Comparing this last expression with the image of the right side of 
\eqref{re-expressed-simplicial-coboundary-map}
under the isomorphism $\varphi$ described in \eqref{definition-of-cochain-isomorphism},
we see that they are equal using the following equality for $\ell=1,2,\ldots,m$:
$$
\sgn(\kappa(i_1),\ldots,\kappa(i_p),j_1,\ldots,j_m) 
= (-1)^{\ell-1} \cdot 
\sgn(\kappa(i_1),\ldots,\kappa(i_p),j_\ell,j_1,\ldots,\widehat{j_\ell},\ldots,j_m).
\qedhere
$$
\end{proof}

\begin{remark}\rm
  \label{balanced-Cohen-Macaulay}
  When $\Delta$ is Cohen-Macaulay over $\k$ and has a 
  balanced $d$-coloring $\kappa$, then $\k[\Delta]$
  will be a free $A$-module, and $\Tor_m^A(\k[\Delta],\k)$ vanishes except for $m=0$.
  As in \eqref{Cohen-Macaulay-poset-simplification},
  one then has this interpretation for each $S$ and
  $\boldb=\sum_{j \in S} \epsilon_j\in \{0,1\}^d$:
  \begin{equation}
\label{balanced-Cohen-Macaulay-simplification}
[h^\kappa_S(\Delta)]=[\tilde{H}^{\#S-1}(\Delta|_S,\k)]
=\Tor_0^A(\k[\Delta],\k)_\boldb.
\end{equation}
This applies, for example, in each of Examples~\ref{Stanley-example-type-A}, ~\ref{Stanley-example-type-B}, and
~\ref{Athanasiadis-example}.

\end{remark}

\section{Simplicial posets and their face rings}
\label{simplicial-poset-section}

As mentioned in the Introduction, the second part of this paper
deals not only with Stanley-Reisner rings of simplicial complexes, but more generally 
with Stanley's {\it face rings} of {\it simplicial posets}, which we review here; see Stanley \cite{Stanley-fvec-hvec} or \cite[\S III.6]{Stanley-littlegreenbook} for
more background.

\begin{defn} \rm
A {\it simplicial poset} $P$ is a poset with (unique) bottom element $\varnothing$, in which every lower interval $[\varnothing,x]:=\{y \in P: \varnothing \leq_P y \leq_P x\}$ is isomorphic to a Boolean algebra.
\end{defn}

\begin{remark} \rm
Bj\"orner \cite[\S2.3]{Bjorner-CW-posets} 
called these {\it posets of Boolean type}, and Garsia and Stanton \cite{Garsia-Stanton} called them {\it Boolean complexes}.
\end{remark}

To each simplicial poset $P$ there is an associated regular $CW$-complex $\Delta$ that has $P$ as its poset of faces, with bottom element $\varnothing$ corresponding to the empty face.  For this reason, we will call
a typical element of $P$ by $F$. 
Stanley associated the following two rings to $\Delta$ or $P$.

\begin{defn} 
\label{Stanley-ring-definition}
\rm
Given a simplicial poset $P$ or its corresponding cell complex $\Delta$,
let $\k[\Delta]$ be the quotient of the polynomial
ring $\k[y_F]$ having a variable for each $F$ in $P$, by the ideal with these generators:
\begin{enumerate}
    \item[(a)] $y_F y_{F'}$ if $F, F'$ have no upper bounds in $P$, and
    \item[(b)] $y_F y_{F'}-y_{F \wedge F'} 
    \sum_{G} y_G$ where the sum is over all the minimal upper bounds $G$ for $F, F'$ in $P,$
    \item[(c)] $y_{\varnothing}-1$.
\end{enumerate}
Let $\tilde{\k}[\Delta]$ be the quotient of $\k[y_F]$
by only (a),(b) above, but not (c),
so that $\k[\Delta] = \tilde{\k}[\Delta]/(y_\varnothing -1)$.
In \cite{Stanley-fvec-hvec}, these two rings $\tilde{\k}[\Delta]$ and $\k[\Delta]$ are denoted
$\hat{A}_P$ and $A_P$.
\end{defn}

\begin{remark} \rm
  As pointed out by Stanley \cite[p. 323]{Stanley-fvec-hvec}, when $\Delta$ happens to be an abstract simplicial complex (equivalently, its
  face poset $P$ is a meet-semilattice), the map $y_F \longmapsto x^F:=\prod_{i \in F} x_i$
  induces a ring isomorphism between the face ring $\k[\Delta]$ just defined
  and the Stanley-Reisner ring of $\Delta$ defined earlier, also called $\k[\Delta]$.  Thus the two seemingly conflicting terminologies are actually compatible. Readers interested only in Stanley-Reisner rings can safely substitute $y_F=\boldx^F$ in all ensuing discussion.
\end{remark}

\begin{remark} \rm
Brun and R\"omer \cite[\S 4]{BrunRomer}
define an interesting extension of face rings $\k[\Delta]$ beyond simplicial posets, to what they call {\it locally distributive lattices}.
\end{remark}

\begin{example} \rm
\label{Stanley-ring-example-for-injective-words-on-[2]}
One of our motivating examples of a simplicial poset is the {\it poset $P$ of injective words on $[n]$} discussed in Example~\ref{Athanasiadis-example}.
Its associated regular CW-complex $\Delta$
is called the {\it complex of injective words on $[n]$},
shown here for $n=2$, along with the face poset $P$.

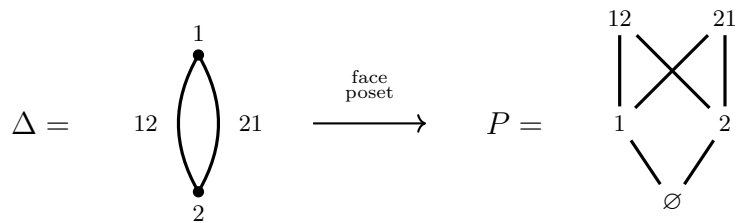
\begin{center}
\begin{tikzpicture}[scale=0.7]
\pgfmathsetlengthmacro{\rad}{25pt};
\pgfmathsetlengthmacro{\vertsize}{\rad*0.11};

\pgfmathsetmacro{\triangleopacity}{0.4};
\definecolor{dimtwocolor}{RGB}{140,140,200}
    \begin{scope}[shift={(-16.5,0)}]
        \node[] at (0,0){\Large $\Delta=$};
    \end{scope}
    
    \begin{scope}[shift={(-13.5,0)}]
    	\coordinate (1) at (0,1.3);
		\coordinate (2) at (0,-1.3);
		
		\coordinate (12) at (-1,0);
		\coordinate (21) at (1,0);
		
		\draw [very thick, draw=black] (1) to [bend left] (2);
		\draw [very thick, draw=black] (1) to [bend right] (2);
		
		\draw [black, fill=black] (1) circle [radius=\vertsize];
		\draw [black, fill=black] (2) circle [radius=\vertsize];
		
		\node[yshift=3*\vertsize]  at (1)     {\small $1$};
		\node[yshift=-3*\vertsize] at (2)     {\small $2$};
		\node                      at (12)    {\small $12$};
		\node                      at (21)    {\small $21$};
    \end{scope}
    
    \begin{scope}[shift={(-10.5,0)}]
        \node (b) at (-1,0) {};
        \node (e) at (1.5,0)  {};
        \draw[thick, draw=black, ->] 
        (b) 
        edgenode[yshift=5mm]{
            $\substack{
                \text{face}\\ 
                \text{poset}
                }$
            }
        (e);
    \end{scope}
    
    \begin{scope}[shift={(-4.5,0)}]
        \node[] at (-3,0){\Large $P=$};
    
		\node (emptyset) at (0,-1.5) {$\varnothing$};
		
		\node (1)  at (-1,0) {\small $1$};
		\node (2)  at (1,0)  {\small $2$};	
		\node (12) at (-1,2) {\small $12$};		
		\node (21) at (1,2)  {\small $21$};		
		
		\draw [very thick, draw=black] (emptyset) -- (1) -- (12);
		\draw [very thick, draw=black] (1) -- (21);
		\draw [very thick, draw=black] (2) -- (12);
	    \draw [very thick, draw=black] (emptyset) -- (2) -- (21);
    \end{scope}
\end{tikzpicture}
\end{center}

Here one has ring presentations
\addtolength{\jot}{3mm}
\begin{align*}
\tilde{\k}[\Delta]
&=\k[y_{\varnothing},y_1,y_2,y_{12},y_{21}]/(y_{12}y_{21}\,\, , \,\,  y_1y_2-y_{\varnothing} (y_{12}+y_{21})) \\
\k[\Delta]
&=\tilde{\k}[\Delta]/(y_{\varnothing}-1) 
\cong \k[y_1,y_2,y_{12},y_{21}]/(y_{12}y_{21} \,\, , \,\, y_1y_2-(y_{12}+y_{21}))
\end{align*}
\end{example}

\subsection{Two gradings}

There are two kinds of gradings of $\k[\Delta]$ that
will play a key role.  The first is an $\N$-graded ring structure
employed by Stanley.

\begin{defn} \rm
\label{single-grading-as-ring}
 ({\it $\N$-grading as a ring})
One can define \cite[p. 325]{Stanley-fvec-hvec} an
$\N$-grading on the polynomial algebra $\k[y_F]$ by decreeing
$\deg_\N(y_F)$ to be the {\it rank} $\rho(F)$ of the Boolean interval
$[\varnothing,F]$ in $P$, that is, 
$$
\deg_{\N}(y_F):=\rho(F)=1+\dim(F),
$$
when $F$ is regarded as a face of the cell complex $\Delta$.
It is not hard to check from the relations (a),(b),(c) that this
$\N$-grading descends to one on the quotient ring
$\tilde{\k}[\Delta]$, in which the degree $0$
component consists of the subalgebra $\k[y_\varnothing]$
generated by $y_\varnothing$.  This then descends to an
$\N$-grading on the further quotient, the face ring $\k[\Delta]$, where one 
sets $y_\varnothing=1$, in 
which the degree $0$ component is the field $\k$.
\end{defn}

The second grading on the face ring $\k[\Delta]$ 
is related to De Concini, Eisenbud and Procesi's
theory of {\it Hodge algebras} or {\it algebra with straightening law (ASL)} on the poset $P$, defined in \cite{Hodge-algebras}.
Stanley observed that $\tilde{\k}[\Delta]$ is an ASL
on the simplicial poset $P$ which is the face poset of the cell complex $\Delta$.
He showed that this leads to 
{\it standard monomial} bases for
the two rings:
\begin{itemize}
    \item The ring $\tilde{\k}[\Delta]$ has
    as $\k$-basis the monomials 
    $
    \{y_{F_1}^{a_1} \cdots y_{F_\ell}^{a_\ell}: \text{ chains }
        y_{F_1} < \cdots < y_{F_\ell} \text{ in }P
        \}.
    $
    \item Its quotient the face ring $\k[\Delta]$ has as $\k$-basis 
        $
    \{y_{F_1}^{a_1} \cdots y_{F_\ell}^{a_\ell}: \text{ chains }
        y_{F_1} < \cdots < y_{F_\ell} \text{ in }P \setminus \varnothing
        \}.
    $
\end{itemize}

The standard monomial basis leads to the second kind of
grading for $\k[\Delta]$.
\begin{defn}
\label{vector-space-multigrading-on-Stanley-ring}
\rm
 ({\it $\N^d$-grading as a $k$-vector space})
Let $d:=\dim(\Delta)+1$.  Then decree in $\k[\Delta]$ that the $\N^d$-degree of $y_{F_1}^{a_1} \cdots y_{F_\ell}^{a_\ell}$ with 
$y_{F_1} < \cdots < y_{F_\ell}$ in $P \setminus \varnothing$
is the vector $\boldb:=\sum_{i=1}^\ell a_i \epsilon_{\rho(F_i)}$ in $\N^d$.
This gives a $\k$-vector space
decomposition (but {\it not} an $\N^d$-graded ring structure)
$$
\k[\Delta]=\bigoplus_{\boldb \in \N^d}
   \k[\Delta]_{\boldb}
$$
where 
$\k[\Delta]_{\boldb}$ is the $\k$-span of standard
monomials of $\N^d$-degree $\boldb$. Note that the {\it grading specialization}
\begin{equation}
\label{grading-specialization-map}
\begin{array}{rcl}
    \N^d &\longrightarrow &\N\\
    \epsilon_j & \longmapsto & j
    \end{array}
\end{equation}
specializes this $\k$-vector space $\N^d$-multigrading to the earlier
$\N$-grading as a ring.
\end{defn}

\noindent
{\bf Warning:} Unlike the $\N$-grading as a ring, small examples like the one below show that the vector space $\N^d$-grading on $\k[\Delta]$ just defined {\it does not respect} its ring multiplication.

\begin{example} \rm
\label{multigrading-is-not-ring-grading}
The complex $\Delta$ of injective words on $[2]$, considered
in Example~\ref{Stanley-ring-example-for-injective-words-on-[2]},
had this face ring:
$$
\k[\Delta]=
\k[y_1,y_2,y_{12},y_{21}]
/( y_{12} y_{21}, y_1 y_2 - (y_{12}+y_{21})).
$$
Using its $\N^2$-grading as a $\k$-vector space,
the element $\theta_1=:y_1+y_2$ is homogeneous 
with $\deg_{\N^2}(\theta_1)=\epsilon_1$.  However, its square
$$
\theta_1^2=y_1^2+2y_1 y_2+ y_2^2
=y_1^2+2(y_{12}+y_{21})+y_2^2
$$
is inhomogeneous for the $\N^2$-grading, assuming $\k$ does not have characteristic $2$, since
$$
\begin{aligned}
&\deg_{\N^2}(y_1^2)=
\deg_{\N^2}(y_2^2)=2\epsilon_1,\\
&\deg_{\N^2}(y_{12})=
\deg_{\N^2}(y_{21})=2\epsilon_2.
\end{aligned}
$$
\end{example}

\subsection{Comparison with the barycentric subdivision}
\label{subdivision-comparison-section}

For any simplicial poset $P$ with cell complex $\Delta$, there is a close relation
between its face ring $\k[\Delta]$ and the Stanley-Reisner ring $\k[\Sd \Delta]$
for the simplicial complex which is its {\it barycentric subdivision} $\Sd \Delta$,
that is, the {\it order complex} $\Delta( P \setminus \varnothing )$;
see Bj\"orner \cite{Bjorner-CW-posets} for more on the identification of $\Delta(P \setminus \varnothing)$ with  $\Sd \Delta$.  

If $\Delta$ has dimension $d-1$,
then $\Sd \Delta$ is a balanced complex with
vertex $d$-coloring $V=P \overset{\kappa}{\longrightarrow} [d]$
assigning $\kappa(F):=\rho(F)$.
One then has a $\k$-vector space (but not ring) isomorphism sending
\begin{equation}
    \label{Grobner-deformation-map}
    \begin{array}{ccc}
\k[\Delta] & \longrightarrow & \k[\Sd \Delta] \\
\Vert & & \Vert\\
\k[y_F]/J_\Delta & & \k[x_F]/I_{\Sd \Delta}
\end{array}
\end{equation}
where $x_F$ is the variable in the Stanley-Reisner ring 
$\k[\Sd \Delta]$ corresponding to the
barycenter vertex of the face $F$ in $\Delta$,
and $y_F$ is the variable of the face ring 
$\k[\Delta]$ associated to the face $F$ as in Definition~\ref{Stanley-ring-definition}. 
The isomorphism sends the $k$-basis elements
$\{ y_{F_1} y_{F_2} \cdots y_{F_\ell}\}$
of $\k[\Delta]$ indexed by multichains of 
faces $F_1 \leq F_2 \leq \cdots \leq F_\ell$ in $P \setminus \varnothing$ to
the corresponding $k$-basis elements
$\{x_{F_1} x_{F_2} \cdots x_{F_\ell}\}$
of $\k[\Sd \Delta]$.
This map also respects the two $\N^d$-multigradings, 
that is the one for
$\k[\Sd \Delta]$ that comes from its $d$-coloring as a balanced simplicial complex, and the one for
$\k[\Delta]$ from Definition~\ref{vector-space-multigrading-on-Stanley-ring}.

\begin{remark}
\label{Grobner-ASL-remark}
\rm
In fact, this vector space isomorphism  \eqref{Grobner-deformation-map}
is really a {\it Gr\"obner deformation} coming from an ASL structure, as we
now explain.  The face ring $\k[\Delta]$ does {\it not} satisfy the axioms given in \cite[\S1.1]{Hodge-algebras} to be an ASL on $P \setminus \varnothing$.
However, if one considers the {\it opposite} or {\it dual poset}
$P^\opp$ having the same underlying set but $F<_{P^\opp} F'$ if and only if
$F' <_P F$, then $\k[\Delta]$ {\it is} an ASL on $(P\setminus \varnothing)^\opp$ instead\footnote{The issue is as follows. When two incomparable faces $F, F'$ of $\Delta$ have $F \wedge F' =\varnothing$, Definition~\ref{Stanley-ring-definition}(b,c) leads
to a rewriting rule that says $y_F y_{F'}=\sum_G y_G$ where $G$ runs
over all minimal upper bounds for $F,F'$ in $P$.  The ASL axioms would require 
each term in that summation to be divisible by at least one $y_G$ with $G<F,F'$, rather than $G > F,F'$.}.
Since 
$$
\Sd \Delta \cong \Delta( P \setminus \varnothing) \cong \Delta( P \setminus \varnothing)^\opp,
$$
this implies that there is a term ordering on the polynomial rings $\k[y_F]$ and $\k[x_F]$
for which $I_{\Sd \Delta}$ is the initial ideal of $J_\Delta$;
see Conca and Varbaro \cite[\S3.1, Rem. 3.11]{Conca-Varbaro}.
In other words, the $\k$-linear map \eqref{Grobner-deformation-map} is a (square-free)
{\it Gr\"obner deformation}.
\end{remark}

Note that the group $G=\Aut(\Delta)$ of cellular automorphisms of the cell complex $\Delta$ corresponds to the poset automorphisms of $P$,
and color-preserving automorphisms for the balanced $d$-coloring
$\kappa$ of $\Sd \Delta$.  Consequently, $\k[\Sd \Delta]$
and $\k[\Delta]$ have the same (equivariant) $\N^d$-graded Hilbert series
in $\Groth_\k(G)[[t_1,\ldots,t_d]]$
\begin{equation}
    \label{complex-and-subdivision-have-same-hilb}
\begin{aligned}
\Hilb_{\equivariant}( \k[\Delta], \boldt)
=\Hilb_{\equivariant}( \k[\Sd \Delta], \boldt )
=\sum_{S \subseteq [d]} \frac{[f^\kappa_S(\Sd \Delta)] \cdot \boldt^S}{\prod_{j \in S}(1-t_j)} 
  =\frac{1}{\prod_{j=1}^d (1-t_j)}
   \sum_{S \subseteq [d]} [h^\kappa_S(\Sd \Delta)] \cdot \boldt^S,
 \end{aligned}
\end{equation}
where the last two expressions come from \eqref{equivariant-Hilb-expressions}.
Of course, the same holds if one specializes to $\N$-gradings,
for example via the map \eqref{grading-specialization-map}.

\begin{example} \rm
Each of Examples~\ref{Stanley-example-type-A}, \ref{Stanley-example-type-B},
\ref{Athanasiadis-example}
was an order complex $\Delta P$ for a
simplicial poset $P$ with some associated cell complex $\Delta$, with a large symmetry group 
$G=\Aut(\Delta)$:
\begin{itemize}
    \item In Example~\ref{Stanley-example-type-A},
$\Delta$ is an $(n-1)$-dimensional simplex.
\item In Example~\ref{Stanley-example-type-B},
$\Delta$ is the boundary of an $n$-dimensional cross-polytope.
\item In Example~\ref{Athanasiadis-example},
$\Delta$ is the complex of injective words on $[n]$.
\end{itemize}
Consequently in each case $\Delta P=\Sd \Delta$.
Furthermore, in each case $\Delta P$ and $\Delta$ were
Cohen-Macaulay over any field $\k$.  Thus when $\k$ has characteristic zero, since those examples computed explicit
expansions into the classes of simple $\k G$-modules for 
$
\sum_S [h^\kappa_S(\Sd \Delta) ] \boldt^S,
$
using \eqref{complex-and-subdivision-have-same-hilb} 
they also give us such expansions for
$\Hilb_\equivariant(\k[\Delta],\boldt)=
\Hilb_\equivariant(\k[\Sd \Delta],\boldt),
$
or for the $\N$-graded version $\Hilb_\equivariant(\k[\Delta],t)$
after specializing 
via \eqref{grading-specialization-map}.

Let us say a bit more about each example.
In Example~\ref{Stanley-example-type-A}, since $\Delta$
is an $(n-1)$-dimensional simplex, its face ring is simply
the polynomial ring $\k[\Delta]=\k[y_1,\ldots,y_n]$.  In this case, the resulting $\N^n$-graded equivariant Hilbert series \eqref{Stanley-Solomon-type-A-formula-hilb-consequence}
for $\k[y_1,\ldots,y_n]$ specializes to a formula
in $R_\k(S_n)[[t]]$ equivalent to
the well-known {\it Lusztig-Stanley fake-degree formula} in 
type $A$ from \cite[Prop. 4.11]{Stanley-invariants}:
$$
\Hilb_\equivariant( \k[y_1,\ldots,y_n], t)
=\frac{1}{(1-t)(1-t^2)\cdots(1-t^n)}
\sum_Q [\lambda(Q)] t^{\maj(Q)}.
$$
Here $Q$ in the sum runs over standard Young tableaux with $n$ cells,
and $\maj(Q):=\sum_{i \in \Des(Q)} i$.

In Example~\ref{Stanley-example-type-B}, where $\Delta$
is the boundary complex of an $n$-dimensional cross-polytope, one can check that its face ring is this Stanley-Reisner ring:
\begin{equation}
\label{cross-polytope-face-ring}
\k[\Delta]=k[x^+_1,x^-_1,\,\, x^+_2,x^-_2,\,\, \ldots,\,\, x^+_n,x^-_n]/
(\,\, x^+_i x^-_i\,\, )_{i=1,2,\ldots,n}
\end{equation}
Here the variables$\{x^+_1,x^-_1,\ldots,x^+_n,x^-_n\}$ correspond to the 
vertices $\{+e_1,-e_1,\ldots,+e_n,-e_n\}$ of the
cross-polytope, and an element 
$w$ in the hyperoctahedral group $B_n$
of all signed permutation matrices permutes the
variables just as it permutes the vertices.
In this case, the resulting $\N^n$-graded equivariant Hilbert series \eqref{Stanley-type-B-formula-hilb-consequence}  specializes to a $B_n$-equivariant Hilbert series 
for the cross-polytope Stanley-Reisner ring in \eqref{cross-polytope-face-ring} that appears to be new.

Lastly, in Example~\ref{Athanasiadis-example}, where $\Delta$
is the complex of injective words, specializing Athanasiadis's formula \eqref{Athanasiadis-formula} gives an $S_n$-equivariant description for the face ring $\k[\Delta]$, which was 
our original goal.
\end{example}

\section{Universal parameters and their depth-sensitivity}
\label{universal-parameters-section}

Recall that for a commutative $\k$-algebra $R$ 
of Krull dimension $d$, a {\it system of parameters} 
is a sequence of elements $\Theta=(\theta_1,\ldots,\theta_d)$
in $R$ for which the ring extension
$\k[\Theta]:=\k[\theta_1,\ldots,\theta_d] \hookrightarrow R$ is finite,
meaning that $R$ is finitely generated as a $\k[\Theta]$-module.

Stanley \cite[Lemma 3.9]{Stanley-fvec-hvec} proves that $\k[\Delta]$ is finitely generated as a module over the $\k$-subalgebra generated by its homogeneous component of degree one, and therefore will always contain {\it linear} systems of parameters.  
However, such linear systems of parameters are rarely stable under the symmetries $\Aut(\Delta)$.  Instead
we will work with the following universal parameters that
are invariant under symmetries.

\begin{defn} \rm
Given a simplicial poset $P$ and its associated cell complex $\Delta$, say of dimension $d-1$, call the {\it universal parameters} 
$
\Theta:=(\theta_1,\ldots,\theta_d)
$
the elements defined for $j=1,2,\ldots,d$ as follows: 
$$
\theta_j:=\sum_{\substack{F \in P:\\ \rho(F)=j}} y_F.
$$
In particular, when $\Delta$ is actually a simplicial complex,
so that $\k[\Delta]$ is its Stanley-Reisner ring, then
$$
\theta_j=\sum_{\substack{F \in \Delta:\\ \#F=j}} \boldx^F.
$$
\end{defn}

\begin{prop}
\label{elementary-theta-are-an-hsop}
For any simplicial poset $P$, these $\Theta$ form a system of parameters in $\k[\Delta]$.
\end{prop}
\begin{proof}
As $\k[\Delta]$ is an ASL on $(P\setminus \varnothing)^\opp$, 
this is
De Concini, Eisenbud and Procesi's \cite[Thm. 6.3]{Hodge-algebras}.
\end{proof}

The universal parameters $\Theta=(\theta_1,\ldots,\theta_d)$ for Stanley-Reisner rings and face rings have already appeared repeatedly in the literature. 
We have followed Herzog and Moradi \cite[\S 3]{HerzogMoradi}
in calling them {\it universal};  they used this terminology 
in the case where $\Delta$ is a simplicial complex.
In this case, one may think of $\Theta$ as the (nonzero) images under $\k[\boldx] \twoheadrightarrow \k[\Delta]$ of the {\it elementary symmetric functions}
in the variables $x_1,\ldots,x_n$, which form a well-known 
system of parameters for $\k[\boldx]$.
The parameters $\Theta$ were also considered by D.E. Smith, whose result \cite[Cor. 6.5]{Smith} is a special case of
our next result, Theorem~\ref{depth-sensitivity-theorem}, removing
two extra hypotheses that he assumed:
\begin{itemize}
\item  $\Delta$ is a simplicial complex, not allowing for simplicial posets, and 
\item $\Delta$ is pure.
\end{itemize}

\begin{thm}
\label{depth-sensitivity-theorem}
  For any simplicial poset with cell complex $\Delta$,
  not necessarily pure, the depth of the face ring $\k[\Delta]$
  is detected by the universal parameters $\Theta$ as follows:
$$   
\depth \,  \k[\Delta] =\max\{\delta: (\theta_1,\theta_2,\ldots,\theta_\delta) \text{ forms a regular sequence on }\k[\Delta]\}.
$$
\end{thm}
\begin{proof}
Since $\depth \,  \k[\Delta] $ is the length of the longest
regular sequence of elements in the irrelevant ideal $\k[\Delta]_+$, it will always be bounded below by
the right side in the theorem.  On the other hand,
Duval has shown \cite[Cor. 6.5]{Duval-thesis} that
for a simplicial poset $P$ with
cell complex $\Delta$, denoting its
{\it $i$-skeleton} $\Delta^{(i)}$, one has
$$
\depth \,  \k[\Delta] =\max\{\delta: \Delta^{(\delta-1)} \text{ is Cohen-Macaulay over }\k \}.
$$
The theorem would therefore follow after proving
the following assertion:
\begin{quote}
If $\Delta$ has $\Delta^{(\delta-1)}$ Cohen-Macaulay over $\k$, then $(\theta_1,\theta_2,\ldots,\theta_\delta)$ is a $\k[\Delta]$-regular sequence.
\end{quote}
We prove this assertion by induction on the
cardinality $\#\Delta \setminus \Delta^{(\delta-1)}$.
In the base case, $\Delta=\Delta^{(\delta-1)}$ 
is a Cohen-Macaulay complex and
$\k[\Delta]$ a Cohen-Macaulay ring, so the assertion follows from
Proposition~\ref{elementary-theta-are-an-hsop},
since every system of parameters forms a 
regular sequence.

In the inductive step, pick a {\it maximal} face $F$ in $\Delta \setminus \Delta^{(\delta-1)}$, and let $\hat{P}, \hat{\Delta}$ be the
simplicial poset and cell complex
obtained by removing $F$ from $P, \Delta$.
Maximality of $F$ gives an
exact sequence of $\k$-vector spaces
\begin{equation}
\label{ses-for-depth-sensitivity}
0 \rightarrow (y_F) \longrightarrow \k[\Delta]
\longrightarrow \k[\hat{\Delta}] \rightarrow 0
\end{equation}
where $(y_F)$ is the principal ideal
of $\k[\Delta]$ generated by $y_F$.
Letting $A:=\k[z_1,z_2,\ldots,z_\delta]$,
one can check that 
\eqref{ses-for-depth-sensitivity}
is also a short exact sequence of $A$-modules in
which  $z_i$ acts
\begin{itemize}
\item on $\k[\Delta]$ and on $(y_F)$
as multiplication by $\theta_i$,
and 
\item on $\k[\hat{\Delta}]$ as multiplication by
$
\hat{\theta}_i:=\sum_G y_G,
$
with the sum  over elements $G$ in $\hat{P}$ having $\rho(G)=i$.
\end{itemize}
We wish to show that $\k[\Delta]$ is a free $A$-module,
since in this graded setting, it is equivalent to
$(\theta_1,\theta_2,\ldots,\theta_\delta)$ forming a 
$\k[\Delta]$-regular sequence.
By induction, $\k[\hat{\Delta}]$
is free as an $A$-module. 
Since \eqref{ses-for-depth-sensitivity} is short exact, using 
a standard fact about regular sequences \cite[Lem. 6.3]{Smith}, \cite[p. 103, Exer. 14]{Kaplansky}, it
suffices
to check that $(y_F)$ is free as an $A$-module.

Assume $F$ has vertex variables
$y_1,\ldots,y_m$, so $m \geq \delta$.
Since $F$ is a maximal face of $\Delta$,
in $\k[\Delta]$ one has
$$
y_F\cdot y_G=
\begin{cases}
0 &\text{ if }G\text{ is not a subface of }F,\\
y_F \cdot \prod_{i \in G} y_i &\text{ if }G\text{ is a subface of }F.\\
\end{cases}
$$
Consequently, the $\k$-linear map defined by 
$$
\begin{array}{rcl}
\k[\mathbf{x}]:=\k[x_1,\ldots,x_m]&  \longrightarrow&  (y_F)\\
x_1^{a_1} \cdots x_m^{a_m} & 
\longmapsto & y_F \cdot y_1^{a_1} \cdots y_m^{a_m}
\end{array}
$$
is an isomorphism of $\k$-vector spaces.  It is also an isomorphism of $A$-modules if
one lets $z_i$ act on $\k[\mathbf{x}]$
via multiplication by the $i^{th}$
elementary symmetric function $e_i(\mathbf{x}):=e_i(x_1,\ldots,x_m)$
for $i=1,2,\ldots,\delta$.  Since these are a subset
of the system of parameters $e_1(\mathbf{x}),\ldots,e_m(\mathbf{x})$ on the Cohen-Macaulay ring $\k[\mathbf{x}]$, then
$e_1(\mathbf{x}),\ldots,e_\delta(\mathbf{x})$ form a regular sequence, and $\k[\mathbf{x}]$
is free as an $A$-module.  Hence $(y_F)$ is also free as an $A$-module.
\end{proof}

\begin{remark} \rm
Results like Theorem~\ref{depth-sensitivity-theorem}
are reminiscent of the role played by $\Theta$ 
in the combinatorial topological
approach to {\it invariant theory} for subgroups $G$ of the symmetric group ${\mathfrak S}_n$ 
acting on $\Q[x_1,\ldots,x_n]$
pioneered by Garsia
and Stanton \cite{Garsia-Stanton}. 

From this viewpoint, Theorem~\ref{depth-sensitivity-theorem}
also fits with the ($q$-analogous) invariant theory 
for subgroups $G$ of the 
finite general linear groups $GL_n(\mathbb{F}_q)$
acting on $\F_q[x_1,\ldots,x_n]$.  
There, one has Landweber and Stong's {\it Depth Conjecture} \cite[p. 260]{Landweber-Stong} asserting that the depth of
the invariant ring $\F_q[x_1,\ldots,x_n]^G$ is similarly detected by the sequence of
{\it Dickson polynomials}, which are $GL_n(\mathbb{F}_q)$-invariant
polynomials $q$-analogous to the elementary symmetric functions.
It would be interesting to find  a closer link  
between these results.
\end{remark}

\begin{example} \rm
Theorem~\ref{depth-sensitivity-theorem} is tight in a certain sense, witnessed by the following family of examples; cf. \cite[Example 6.7]{Smith}.
For each $\delta,d$ with $1 \leq \delta \leq d$, define a
simplicial complex $\Delta(d,\delta)$ on $d+1$ vertices
$\{x_0,x_1,\ldots,x_{d+1}\}$ with two maximal faces, 
\begin{itemize}
    \item the larger maximal face $F_1=\{x_1,x_2,\ldots,x_d\}$ of dimension $d-1$, 
    \item the smaller maximal face
$F_2=\{x_0,x_1,x_2,\ldots,x_{\delta-1}\}$ of dimension $\delta-1$, 
\item with intersection
the $(\delta-2)$-face $F_1 \cap F_2=\{x_1,x_2,\ldots,x_{\delta-1}\}$.
\end{itemize}
Then $\Delta=\Delta(d,\delta)$
has $\k[\Delta]$ of Krull dimension $d$
and depth $\delta$.  Theorem~\ref{depth-sensitivity-theorem} shows
that $\theta_1,\ldots,\theta_\delta$ form a regular sequence.
One can check that each $\theta_j$ for $j=\delta+1,\delta+2,\ldots,d$
is a nonzero element of $\k[\Delta]$, but a zero-divisor, since these
$\theta_j$ are annihilated by multiplication with the (nonzero) element $x_0$.
Thus the {\it only}
 subsets of $\{\theta_1,\theta_2,\ldots,\theta_d\}$
which form $\k[\Delta]$-regular sequences are exactly the
subsets of $\{\theta_1,\theta_2,\ldots,\theta_{\delta}\}$.
\end{example}

\section{A conjecture on resolving over the
universal parameters}
\label{main-conjecture-section}

Given a simplicial poset $P$ with cell complex $\Delta$,
Proposition~\ref{elementary-theta-are-an-hsop}
shows that the face ring $\k[\Delta]$ is a finitely-generated module over the universal parameter ring
$\k[\Theta]=\k[\theta_1,\ldots,\theta_d]$.
It therefore makes sense to consider the minimal finite
free $k[\Theta]$-resolution of $\k[\Delta]$, and
compute $\Tor^{\k[\Theta]}(\k[\Delta],\k)$.

We should be slightly careful about the structures
carried by these objects.  Because the $\k[\Theta]$-module
structure on $\k[\Delta]$ comes from its ring structure,
it preserves the $\N$-grading on $\k[\Delta]$ as a ring
described in Definition~\ref{single-grading-as-ring},
assuming that $\deg(\theta_j):=j$ in $\k[\Theta]$,
as one would expect.  However, the $\k[\Theta]$-module
structure on $\k[\Delta]$ does {\it not} respect
the $\N^d$-grading as a $\k$-vector space
described in Definition~\ref{vector-space-multigrading-on-Stanley-ring}.  This has been illustrated already by small examples such as
Example~\ref{multigrading-is-not-ring-grading}, in which $\theta_1$ is homogeneous for the $\N^2$-grading, while
$\theta_1^2$ is inhomogeneous.

Hence we will only consider $\N$-graded free $k[\Theta]$-resolutions
of $\k[\Delta]$.  Also, note that 
each of the universal parameters $\theta_j$ is fixed by the group $\Aut(\Delta)$, and hence this
group action commutes with the $\k[\Theta]$-module structure on $\k[\Delta]$, preserving the $\N$-grading.  
Using Proposition~\ref{equivariant-resolution-prop}, one can produce a group equivariant free $k[\Theta]$-resolution, and $\Aut(\Delta)$ also
acts on each $\k$-vector space $\Tor_m^{\k[\Theta]}(\k[\Delta],\k)_j$.  

Conjecture~\ref{canonical-resolution} below describes
$\Tor^{\k[\Theta]}(\k[\Delta],\k)$ by comparing $\k[\Delta]$ 
with the Stanley-Reisner ring 
$\k[\Sd \Delta]$ 
for the barycentric subdivision,
as discussed in Section~\ref{subdivision-comparison-section}.
The $k$-vector space isomorphism
$\k[\Delta] \longrightarrow \k[\Sd \Delta]$
and Gr\"obner deformation in
  \eqref{Grobner-deformation-map}
  sends the {\it universal parameter ring}
  $$
  \k[\Theta]=\k[\theta_1,\ldots,\theta_d] \subset \k[\Delta]
  $$
  inside the face ring of $\Delta$
  to the {\it colorful parameter ring} $$\k[\Gamma]=\k[\gamma_1,\ldots,\gamma_d] \subset \k[\Sd \Delta]$$
  inside the Stanley-Reisner ring of $\Sd \Delta$,
  where the colorful parameters come from $\Sd \Delta = \Delta (P \setminus \varnothing)$ being a balanced $(d-1)$-dimensional
  simplicial complex.  
The colorful Hochster formula Theorem~\ref{colorful-Hochster-formula}
then describes the $\N^d$-graded vector space
 $\Tor^{\k[\Gamma]}(\k[\Sd \Delta],\k)$
 in an equivariant
 fashion, while Conjecture~\ref{canonical-resolution}
 specializes this to an $\N$-grading via the map in \eqref{grading-specialization-map}
 to describe $\Tor^{\k[\Theta]}(\k[\Delta],\k)$
 equivariantly.
 
\begin{conjecture}
\label{canonical-resolution}
For any simplicial poset with associated cell complex $\Delta$ of dimension $d-1$, and any subgroup $G$ of $\Aut(\Delta)$, for each $m=0,1,\ldots,d$ one has
these equalities in $R_\k(G)$:
$$
\left[ \Tor_m^{\k[\Theta]}(\k[\Delta],\k)_j \right]
= \left[ \Tor_m^{\k[\Gamma]}(\k[\Sd\Delta],\k)_j \right] 
=
\sum_{\substack{S\subseteq [d]:\\ j=\sum_{s\in S} s}}
\left[ \tilde{H}^{\#S-m-1}\left( \left( \Sd\Delta \right)|_S , \k\right) \right]
$$
Equivalently, one has this equality in $R_\k(G)[[t]]$:
\begin{equation}
    \label{canonical-resolution-rephrased}
\Hilb_\equivariant( \Tor_m^{\k[\Theta]}(\k[\Delta],\k), t )
=
\Bigg[ \Hilb_\equivariant( \, \Tor_m^{\k[\Gamma]}(\k[\Sd\Delta],\k), \,\,  t_1,\ldots,t_d )
\Bigg]_{\substack{t_1=t\\t_2=t^2\\\vdots\\t_d=t^d}}.
\end{equation}
\end{conjecture}

\begin{remark} \rm
When $\k G$ is semisimple,
the first line of equalities in the conjecture would be isomorphisms:
$$
\Tor_m^{\k[\Theta]}(\k[\Delta],\k)_j 
\cong \Tor_m^{\k[\Gamma]}(\k[\Sd\Delta],\k)_j 
=
\bigoplus_{\substack{S\subseteq [d]:\\ j=\sum_{s\in S} s}}
\tilde{H}^{\#S-m-1}\left( \left( \Sd\Delta \right)|_S , \k\right)
$$
This happens, e.g., if one ignores the group action by taking $G=\{1\}$, or more generally, when $\#G \in \k^\times$.
\end{remark}

\begin{remark} \rm
After posting a version of this paper to the {\tt arXiv}, S. Murai suggested the following question about an even stronger assertion than Conjecture~\ref{canonical-resolution}:

\begin{question} \rm 
    \label{Murai's-question}
    Regard the universal parameters $\Theta$ and the colorful parameters $\Gamma$ as generating the same subalgebra $A=\k[z_1,\ldots,z_d]$ of the polynomial ring $\k[y_F]_{\varnothing \neq F \in \Delta}$, where $z_i:=\sum_{\substack{F \in \Delta:\\\rho(F)=i}} y_F$. Does there exist an isomorphism of ($\N$-graded) $A$-modules
    $
    \k[\Delta] \cong \k[\Sd\Delta]?
    $
    Is there an equivariant isomorphism?
\end{question}
In all examples that we have checked so far, the answer is ``yes."
\end{remark}

\begin{example}\rm
The balanced simplicial complex considered in Example~\ref{running-example-one} is actually the barycentric subdivision $\Sd \Delta$ for this regular cell complex $\Delta$ coming from a simplicial poset:

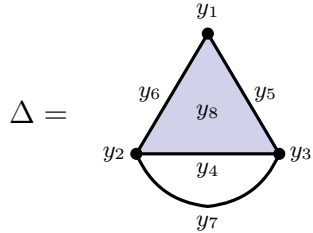
\begin{center}
\begin{tikzpicture}[scale=0.5]
\pgfmathsetlengthmacro{\rad}{25pt};
\pgfmathsetlengthmacro{\vertsize}{\rad*0.1666};

\pgfmathsetmacro{\triangleopacity}{0.4};
\definecolor{dimtwocolor}{RGB}{140,140,200}
    
    \begin{scope}[shift={(-4.5,0)}]
        \node[] at (0,0){\Large $\Delta=$};
    \end{scope}
     
    \begin{scope}{[shift={(0,0)}]}
    
     	\coordinate (1) at (0,2.1);
		\coordinate (2) at (-1.9,-1.1);
		\coordinate (3) at (1.9,-1.1);
		
     	\coordinate (4) at (0,-1.1);
		\coordinate (5) at (.95,.5);
		\coordinate (6) at (-.95,.5);
		\coordinate (7) at (0,-2.5); 
		\coordinate (8) at (0,0); 
		
		\draw [very thick, draw=black, fill=dimtwocolor, fill opacity=\triangleopacity] (1) -- (6) -- (2) -- (4) -- (3) -- (5) -- cycle;

		\draw [very thick, draw=black] (2) to [bend right] (7) to [bend right] (3);

		\draw [black, fill=black] (1) circle [radius=\vertsize];
		\draw [black, fill=black] (2) circle [radius=\vertsize];
		\draw [black, fill=black] (3) circle [radius=\vertsize];
		
		\node[yshift=2*\vertsize]               at (1) {\small $y_1$};
		\node[xshift=-2*\vertsize]              at (2) {\small $y_2$};		
		\node[xshift=2*\vertsize]               at (3) {\small $y_3$};
		\node[yshift=-1.5*\vertsize]            at (4) {\small $y_4$};
		\node[xshift=2*\vertsize]               at (5) {\small $y_5$};		
		\node[xshift=-2*\vertsize]              at (6) {\small $y_6$};
		\node[yshift=-1.5*\vertsize]            at (7) {\small $y_7$};
		\node[]                                 at (8) {\small $y_8$};
    \end{scope}
\end{tikzpicture}
\end{center}

\noindent
We examine the free $\k[\Theta]$-resolution of $\k[\Delta]$,
where the universal parameter ring
$\k[\Theta]=\k[\theta_1,\theta_2,\theta_3]$,
has
$$
\theta_1 = y_1+y_2+y_3,\quad
\theta_2 = y_4+y_5+y_6+y_7, \quad
\theta_3 = y_8.
$$
\newpage
Here is the {\tt Macaulay2} output:
\begin{verbatim}
i1 : S = QQ[y_1..y_8, Degrees=>{1,1,1,2,2,2,2,3}];

i2 : IDelta = ideal(y_1*y_2-y_6, y_1*y_3-y_5, y_1*y_4-y_8, y_1*y_7, 
                    y_2*y_3-(y_4+y_7), y_2*y_5-y_8, y_3*y_6-y_8,
                    y_4*y_5-y_3*y_8,y_4*y_6-y_2*y_8, y_4*y_7, 
                    y_5*y_6-y_1*y_8, y_5*y_7, y_6*y_7, y_7*y_8);

i3 : phi = map(S, QQ[z_1..z_3,Degrees=>{1,2,3}], matrix{{y_1+y_2+y_3, y_4+y_5+y_6+y_7, y_8}});

i4 : betti res pushForward(phi, S^1/IDelta);

            0 1
o4 = total: 8 2
         0: 1 .
         1: 2 .
         2: 3 .
         3: 2 .
         4: . 1
         5: . 1
\end{verbatim}

\noindent
Conjecture~\ref{canonical-resolution} says this
could have been obtained from the 
equivariant $\N^3$-graded Betti table \eqref{equivariant-multigraded-Betti-table-example} for the $\k[\Gamma]$-resolution of $\k[\Sd \Delta]$ appearing in Example~\ref{example-revisited}.
One first applies the $\N^3 \rightarrow \N$ grading specialization map $t_i \mapsto t^i$ from \eqref{grading-specialization-map}, giving these
equivariant
descriptions for $[\Tor_i^{\k[\Theta]}(\k[\Delta],\k)_j]$
in $R_\k(G)=\Z[\epsilon]/(\epsilon^2-1)$:
$$
\begin{tabular}{|c|c|c|}\hline
$j$ & $[\Tor_0^A(\k[\Delta],\k)_j]$ & $[\Tor_1^A(\k[\Delta],\k)_j]$\\\hline\hline
$0$& $1$&  \\ \hline
$1$& $1+\epsilon$ & \\\hline
$2$& $2+\epsilon$ & \\\hline
$3$& $2\epsilon$ &  \\\hline
$4$& & \\\hline
$5$&   &  $1$ \\\hline
$6$&  & $\epsilon$  \\\hline
\end{tabular}
$$
Then applying the dimension homomorphism $\epsilon \mapsto 1$ from
\eqref{dimension-homomorphism} gives the above Betti table from {\tt Macaulay2}.
\end{example}

We close with various bits of evidence for
Conjecture~\ref{canonical-resolution}.

\begin{prop}
\label{hilb-prediction-correct-prop}
Conjecture~\ref{canonical-resolution} 
 predicts the correct $\N$-graded
equivariant Hilbert series for $\k[\Delta]$.
\end{prop}
\begin{proof}
Applying the grading specialization $\N^d \rightarrow \N$ map
\eqref{grading-specialization-map} to the equality in  \eqref{complex-and-subdivision-have-same-hilb}
shows that 
\begin{equation}
\label{hilbs-specialize}
\Hilb_{\equivariant}(\k[\Delta],t)
=\left[\Hilb_{\equivariant}(\k[\Sd \Delta],\boldt)\right]_{t_j=t^j}
\end{equation}
On the other hand, since both $\k[\Theta]$ and $\k[\Gamma]$
have trivial $G$-action and the same $\N$-graded Hilbert series
$1 /(1-t)(1-t^2)\cdots (1-t^d)$,  then
by using Conjecture~\ref{canonical-resolution} in the form of \eqref{canonical-resolution-rephrased},
and taking an alternating sum on $m$
as in \eqref{general-resolution-hilb-expression}, one deduces this same
equality \eqref{hilbs-specialize} .
\end{proof}

\begin{cor}
\label{Cohen-Macaulay-corollary}
Conjecture~\ref{canonical-resolution} is correct when
$\k[\Delta]$ is Cohen-Macaulay.
\end{cor}
\begin{proof}
When $\k[\Delta]$ is Cohen-Macaulay, it is a free $\k[\Theta]$-module,
so only $\Tor_0^A(\k[\Delta],\k)$ is non-vanishing,
and the rephrased version \eqref{canonical-resolution-rephrased} of the conjecture
is equivalent to the known equation \eqref{hilbs-specialize}. 
\end{proof}

\begin{prop}
 Conjecture~\ref{canonical-resolution} is correct when $\Delta$ is a $1$-dimensional complex, that is, a graph
 with multiple edges allowed, but no self-loops.
\end{prop}
\begin{proof}[Proof sketch.]
We omit the full details, which are slightly tedious.
Note that since $\Delta$ is a graph, so that $\k[\Theta]=\k[\theta_1,\theta_2]$, 
one knows that $\Tor^{\k[\Theta]}_m(\k[\Delta,\k)$ vanishes for $m \geq 2$.
Hence \eqref{general-resolution-hilb-expression} says here that
$$
\Hilb_\equivariant(\k[\Delta],t)
= \Hilb(\k[\Theta],t) \cdot \left( 
\Hilb_\equivariant(\Tor_0^{\k[\Theta]}(\k[\Delta],\k),t)
-\Hilb_\equivariant(\Tor_1^{\k[\Theta]}(\k[\Delta],\k),t)
\right).
$$
Since Proposition~\ref{hilb-prediction-correct-prop}
says Conjecture~\ref{canonical-resolution} correctly describes $\Hilb_\equivariant(\k[\Delta],t)$, it suffices to check
that the conjecture
correctly describes  $\Tor_0(\k[\Delta],\k)$, and then it must
also correctly describe $\Tor_1(\k[\Delta],\k)$.

Proceed by reformulating
$$
\Tor_0^{\k[\Theta]}(\k[\Delta],\k) \cong \k[\Delta]/(\Theta)=
\k[\Delta]/(\theta_1,\theta_2).
$$
One can then use a part of De Concini, Eisenbud and Procesi's
result \cite[Thm 6.3]{Hodge-algebras}: not only is
$\Theta$ a system of parameters for $\k[\Delta]$,
but $\k[\Delta]$ is generated as a $\k[\Theta]$-module by the
standard monomials $\{y_{F_1} y_{F_2} \cdots y_{F_\ell}\}$
in which $F_1 \lneqq F_2 \lneqq \cdots \lneqq F_\ell$,
that is, where the chain of faces $\{F_i\}_{i=1}^\ell$ has {\it no repeats}.
In particular, when $\Delta=(V,E)$ is a graph with vertices $V$ and
edges $E$, the homogeneous components $\left(\k[\Delta]/(\Theta)\right)_j$ for $j=0,1,2,3$ are $\k$-spanned, respectively by
the images of these sets of monomials
$$
\{1\}, \quad
\{y_v\}_{v\in V}, \quad
\{y_e\}_{e\in E}, \quad
\{y_v y_e\}_{\substack{v \in V, e \in E\\ v < e}},
$$
and $\left(\k[\Delta]/(\Theta)\right)_j=0$ for $j \geq 4$.
This lets one write down four equivariant isomorphisms (details omitted):
$$
\begin{aligned}
\tilde{H}^{-1}(\Sd \Delta|_{\varnothing},\k) &\cong \k \cong (\k[\Delta]/(\Theta))_0 \\
\tilde{H}^{0}(\Sd \Delta|_{\{1\}},\k)&\cong  (\k[\Delta]/(\Theta))_1 \\
 \tilde{H}^{0}(\Sd \Delta|_{\{2\}},\k) &\cong  (\k[\Delta]/(\Theta))_2 \\
\tilde{H}^{1}(\Sd \Delta|_{\{1,2\}},\k)  &\cong   (\k[\Delta]/(\Theta))_3
\end{aligned}
$$
These isomorphisms show
Conjecture~\ref{canonical-resolution} 
correctly describes 
$\Tor^{\k[\Theta]}_0(\k[\Delta],\k)$, completing the proof.
\end{proof}

\begin{prop}
\label{conj-inequality-is-right}
Ignoring group actions,
Conjecture~\ref{canonical-resolution}
gives a correct dimension upper bound:
$$
\dim_\k \Tor_m^{\k[\Theta]}(\k[\Delta],\k)_j
\leq  
\dim_\k \Tor_m^{\k[\Gamma]}(\k[\Sd \Delta],\k)_j.
$$
\end{prop}
\begin{proof}[Proof sketch.]
This requires a variant on the proof of the standard fact (as in Herzog \cite[Thm. 3.1]{Herzog}) that for a polynomial ring $S$,
the graded Betti numbers in a minimal $S$-free resolution of a graded quotient $S/I$ can only increase under Gr\"obner deformations $S/J \rightarrow S/I$, like the map $\k[\Delta] \rightarrow \k[\Sd \Delta]$
in \eqref{Grobner-deformation-map}.  One needs a version
that allows for resolutions of $S/J, S/I$ over a smaller polynomial subalgebra
$\k[\Theta]=\k[\theta_1,\ldots,\theta_d] \subset S$.
To alter the proof of \cite[Thm. 3.1]{Herzog}, consider
$
\k[\Theta,t] \subset \tilde{S}:=S[t]
$
and a minimal graded free $\k[\Theta,t]$-resolution of $\tilde{S}/\tilde{J}$, rather than a free $\tilde{S}$-resolution.  The rest proceeds as before.
\end{proof}

\begin{remark} \rm
    Assuming that $\k G$ is semisimple, then
   Proposition~\ref{conj-inequality-is-right} can be strengthened
    to say that $\Tor_m^{\k[\Theta]}(\k[\Delta],\k)_j$ is a subquotient of $\Tor_m^{\k[\Gamma]}(\k[\Sd \Delta],\k)_j$ as a $\k G$-module (and hence, by semisimplicity, also a $k G$-submodule).  The proof requires further technicalities, so we omit it here. When $\k G$ is {\it not} semisimple, we do not know if it is always a
   subquotient.
\end{remark}

\begin{remark} \rm \
The {\tt Macaulay2} code used in the development of this paper is now available as the package  {\tt ResolutionsOfStanleyReisnerRings}, written by the first author.
\end{remark}

\section*{Acknowledgements}

The authors thank Nathan Nichols and Mahrud Sayrafi for help
with computations in {\tt Macaulay2}.  They also thank Patricia Klein for helpful discussions regarding reference \cite{Conca-Varbaro}, and thank Satoshi Murai for suggesting Question~\ref{Murai's-question}.  Lastly, they thank an anonymous referee for helpful suggestions.


\end{document}